\documentclass{siamltex}
\usepackage{amsmath}
\usepackage{amsfonts}
\usepackage{sw20siam}
\usepackage{subfigure}
\usepackage{graphicx}
\usepackage[colorlinks=true, bookmarksopen,
            pdfauthor={CST},
            pdfcreator={pdftex},
            pdfsubject={algorithms},
            linkcolor={blue},
            anchorcolor={black},
            citecolor={blue},
            filecolor={magenta},
            menucolor={black},
            pagecolor={blue},
            plainpages=false,pdfpagelabels,
            urlcolor={blue} ]{hyperref}

\newtheorem{remark}[theorem]{Remark}

\input{tcilatex}
\begin{document}

\title{	Analysis of the effect of Time Filters on the implicit method: increased accuracy and improved stability}
\author{Ahmet Guzel\thanks{%
Department of \ Mathematics, University of \ Pittsburgh, Pittsburgh, PA,
15260, USA; ahg13@pitt.edu; Partially supported by NSF Grant DMS 1522574}
and \and  William Layton\thanks{%
Department of \ Mathematics, University of \ Pittsburgh, Pittsburgh, PA,
15260, USA; wjl@ pitt.edu, http://www.math.pitt.edu/\symbol{126}wjl;
Partially supported by NSF Grant DMS 1522267 and NSF CBET-CDS\&E grant
1609120}}
\date{11 October 2005}
\maketitle

\begin{abstract}
This report considers linear multistep methods through time filtering. The
approach has several advantages. It is modular and requires the addition of
only one line of additional code. Error estimation and variable timesteps is
straightforward and the individual effect of each step\ is conceptually
clear. We present its development for the backward Euler method and a
curvature reducing time filter leading to a 2-step, strongly A-stable,
second order linear multistep method.
\end{abstract}

\keyphrases{time filter, linear multistep method}
\AMclass{1234.56}

\section{Introduction} \hfill \\

The fully implicit/backward Euler method is commonly the first method
implemented when extending a code for the steady state problem and often the
method of last resort for complex applications. The issue can then arise of
how to increase numerical accuracy in a complex, possibly legacy code
without implementing from scratch another, better method. We show herein
that adding one line, a curvature reducing time filter, increases accuracy
from first to second order, gives an immediate error estimator and induces a
method akin to BDF2. In the 2-step combination the effect of each step is
conceptually clear and immediately adapts to variable timesteps.

To begin, consider the initial value problem%
\begin{equation*}
y^{\prime }(t)=f(t,y(t)),y(0)=y_{0}.
\end{equation*}%
Denote the $n^{th}$ timestep by $k_{n}$. Let $t_{n+1}=t_{n}+k_{n}$, $\tau
=k_{n}/k_{n-1},$ $\nu $ be an algorithm parameter and $y_{n}$ an
approximation to $y(t_{n})$. Discretize this by the standard backward Euler
(fully implicit) method followed by a simple time filter (next for constant
timestep)%
\begin{equation}
\begin{array}{ccc}
\text{Step 1} & : & \frac{y_{n+1}-y_{n}}{k}=f(t_{n+1},y_{n+1}), \\ 
&  &  \\ 
\text{Step 2} & : & y_{n+1}\Leftarrow y_{n+1}-\frac{\nu }{2}\left\{
y_{n+1}-2y_{n}+y_{n-1}\right\} .%
\end{array}
\label{eq:methodConstantTimestep}
\end{equation}%
Step 2 is the only $3-$point filter for which the combination of backward
Euler plus a time filter produces a consistent approximation. The
combination is second order accurate for $\nu =+2/3$, Proposition 2.1.
Proposition 2.2 establishes that the combination is $0-$stable for $-2\leq
\nu <+2$, unstable otherwise and $A-$stable for $-2/3\leq \nu \leq +2/3$.\
Since Step 2 with $\nu =+2/3$ has greater accuracy than Step 1, the pre- and
post- filter difference 
\begin{equation} \label{Eq:Estimator}
EST=|y_{n+1}^{prefilter}-y_{n+1}^{postfilter}|
\end{equation}%
can be used in a standard way to estimate the error in the method and adapt
the timestep.

The variable timestep case is considered in Section 3 based on a definition
of discrete curvature and a curvature reducing discrete filter, Step 2 in (%
\ref{eq:VariableMethod}):%
\begin{equation}
\begin{array}{ccc}
\text{Step 1} & : & \frac{y_{n+1}-y_{n}}{k_{n}}=f(t_{n+1},y_{n+1}), \\ 
\\
\text{Step 2} & : & y_{n+1}\Leftarrow y_{n+1}-\frac{\nu }{2}\left\{ \frac{%
2k_{n-1}}{k_{n}+k_{n-1}}y_{n+1}-2y_{n}+\frac{2k_{n}}{k_{n}+k_{n-1}}%
y_{n-1}\right\} .%
\end{array}
\label{eq:VariableMethod}
\end{equation}%
For variable timestep, the choice of $\nu $ for second order accuracy
depends on $\tau $ and is $\nu =\frac{\tau (1+\tau )}{(1+2\tau )}$, Proposition 3.3.
The filter step reduces the discrete curvature, Definition 3.1, at the three
points $(t_{n+1},y_{n+1})$, $(t_{n},y_{n})$, $(t_{n-1},y_{n-1})$,
Proposition 3.1, provided
\begin{align*} 
0<\nu <1+k_{n}/k_{n-1}.
\end{align*}

For constant time step, the special value $\nu =2/3$ induces a one-leg, two
step method\footnote{%
This $2-$step method seems to be new\ in the sense that, while for the
special value $\nu =2/3$ the LHS is the same as BDF2 for both constant and
variable timesteps, the RHS, as well as the approach to \ implementation,
seems to be new.} that is second order accurate and strongly $A-$stable,
given by%
\begin{equation}
\frac{3}{2}y_{n+1}-2y_{n}+\frac{1}{2}y_{n-1}=kf(t_{n+1},\frac{3}{2}%
y_{n+1}-y_{n}+\frac{1}{2}y_{n-1}).  \label{eq:LMMconstant}
\end{equation}%
\ For general $\nu $ and variable timestep the equivalent linear multistep
method is (\ref{eq:LMMvariable}). The LHS of (\ref{eq:LMMconstant}) is the
same as BDF2. The RHS differs from BDF2 by 
\begin{equation*}
\frac{3}{2}y(t_{n+1})-y(t_{n})+\frac{1}{2}y(t_{n-1})=\ y(t_{n+1})+\mathcal{O}%
(k^{2}),
\end{equation*}%
as required for second order accuracy.

\begin{remark}
The filter value $\nu =2$ in (\ref{eq:methodConstantTimestep}) is not good
since it forces $y_{n+1}$ to be the linear extrapolation of $y_{n},y_{n-1}$.
Thus we always assume $\nu \neq 2.$ The filter can also be repeated several
times (but not iterated to convergence). Filtering twice is equivalent to
increasing the value of the filter parameter $\nu \rightarrow \nu (2-\frac{%
\nu }{2})$ and filtering once.

Time filters centered at $t_{n}$ rather than $t_{n+1}$, are often used in
geophysical fluid dynamics simulations with the leapfrog integrator to
reduce oscillations in the computed solution, Asselin \cite{A72}, Robert \cite{R69}, Williams \cite%
{W11}. As a related example, the Robert-Asselin filter is commonly used and
given by%
\begin{equation*}
y_{n}\Leftarrow y_{n}+\frac{\nu }{2}\left\{ y_{n+1}-2y_{n}+y_{n-1}\right\}
,\nu \simeq 0.1.
\end{equation*}%
The extension of the RA filter to variable timesteps based on Section 3.1 is%
\begin{equation*}
y_{n}\Leftarrow y_{n}+\frac{\nu }{2}\left\{ \frac{2k_{n-1}}{k_{n}+k_{n-1}}%
y_{n+1}-2y_{n}+\frac{2k_{n}}{k_{n}+k_{n-1}}y_{n-1}\right\} .
\end{equation*}%
For a one step method, filters centered at $t_{n},$ like the Robert-Asselin
filter, postprocess the computed solution but do not alter the evolution of
the approximate solution. For that reason the filter is shifted to $t_{n+1}$
herein.
\end{remark}

\section{Constant timestep} \hfill \\

We develop the properties of the method for constant time step in this
section.

\subsection{Derivation of the method} 

Denote the pre-filtered value $y_{n+1}^{\ast }$. Consider backward Euler
plus a general, $3-$point time filter%
\begin{equation}
\begin{array}{ccc}
\text{Step 1} & : & \frac{y_{n+1}^{\ast }-y_{n}}{k}=f(t_{n+1},y_{n+1}^{\ast}), \\ 
\\
\text{Step 2} & : & y_{n+1}=y_{n+1}^{\ast }+\{ay_{n+1}^{\ast
}+by_{n}+cy_{n-1}\}.%
\end{array}%
\end{equation}%
Eliminating the intermediate value $y_{n+1}^{\ast }$, Steps 1 and 2 induce
an equivalent 2-step method for the post-filtered values.

We prove the following.

\begin{proposition}
Let the time step be constant. The combination backward Euler plus time
filter is consistent if and only if the filter coefficients are:%
\begin{equation*}
a=-\frac{\nu }{2},c=-\frac{\nu }{2},b=\nu ,
\end{equation*}%
for some $\nu \neq 2$, and the filter is thus%
\begin{equation}
\ y_{n+1}=\ y_{n+1}^{\ast }-\frac{\nu }{2}\ \left( y_{n+1}^{\ast }-2\
y_{n}+\ y_{n-1}\right) .  \label{eq:FilterConstantStep}
\end{equation}%
In this case the equivalent $2-$step method is 
\begin{gather}
\frac{1}{1-\frac{\nu }{2}}y_{n+1}-\frac{1+\frac{\nu }{2}}{1-\frac{\nu }{2}}%
y_{n}+\frac{\frac{\nu }{2}}{1-\frac{\nu }{2}}y_{n-1}= \\
=kf(t_{n+1},\frac{1}{1-\frac{\nu }{2}}y_{n+1}-\frac{\nu }{1-\frac{\nu }{2}}%
y_{n}+\frac{\frac{\nu }{2}}{1-\frac{\nu }{2}}y_{n-1}).  \notag
\end{gather}%
The combination of is second order accurate if and only if%
\begin{equation*}
\nu =+\frac{2}{3}.
\end{equation*}
\end{proposition}

\begin{proof}
Eliminating $y_{n+1}^{\ast }$ in Step 1 using 
\begin{equation*}
y_{n+1}^{\ast }=\frac{1}{1+a}\left( y_{n+1}-by_{n}-cy_{n-1}\right)
\end{equation*}%
yields an equivalent one-leg linear multistep method for the post-filtered
values 
\begin{equation*}
\frac{1}{1+a}y_{n+1}-\frac{1+a+b}{1+a}y_{n}-\frac{c}{1+a}y_{n-1}=kf(t_{n+1},%
\frac{1}{1+a}y_{n+1}-\frac{b}{1+a}y_{n}-\frac{c}{1+a}y_{n-1}).
\end{equation*}%
In terms of the standard description of a general $2-$step method, the
coefficients are%
\begin{equation*}
\begin{array}{ccc}
\alpha _{2}=\frac{1}{1+a} & , & \beta _{2}=\frac{1}{1+a} \\ 
\alpha _{1}=-\frac{1+a+b}{1+a} & , & \beta _{1}=-\frac{b}{1+a} \\ 
\alpha _{0}=-\frac{c}{1+a} & , & \beta _{0}=-\frac{c}{1+a}.%
\end{array}%
\end{equation*}%
The method is consistent if and only if the first two terms in the method's
LTE expansion are zero and second order accurate if and only if the third
term vanishes. Consistency thus requires%
\begin{equation*}
consistent\Leftrightarrow \left\{ 
\begin{array}{ccc}
&\alpha _{2}+\alpha _{1}+\alpha _{0}=0 \\
& \Updownarrow \\
& a+b+c=0, \\ 
\\
&\alpha _{2}\ -\alpha _{0}\ -(\beta _{2}\ +\beta _{1}\ +\beta _{0}\ )=0 \\
& \Updownarrow \\
& a=c.%
\end{array}%
\right.
\end{equation*}%
Thus for the method to be consistent%
\begin{equation*}
a+b+c=0,a=c,b=-2a
\end{equation*}%
and the first claim follows. The second claim follows by inserting these
values for $a,b,c$.

The condition for second order accuracy is%
\begin{gather*}
\frac{1}{2}\alpha _{2}\ +\frac{1}{2}\alpha _{0}\ -\beta _{2}\ +\beta
_{0}=0\\
\Updownarrow \\
\frac{1}{2}\cdot 1+\frac{1}{2}(-c)-1+(-c)=0\\
\Updownarrow \\
c=-\frac{1}{3}.
\end{gather*}%
These values correspond, as claimed, to $\nu =\frac{2}{3}$ and%
\begin{equation*}
\ y_{n+1}=\ y_{n+1}^{\ast }-\frac{1}{3}\ \left( y_{n+1}^{\ast }-2\ y_{n}+\
y_{n-1}\right) .
\end{equation*}
\end{proof}

\subsection{Stability}

We analyze stability for constant timestep. Consider%
\begin{eqnarray*}
\frac{y_{n+1}^{\ast }-y_{n}}{k} &=&f(t_{n+1},y_{n+1}^{\ast }), \\
\\
y_{n+1} &=&y_{n+1}^{\ast }-\frac{\nu }{2}\left\{ y_{n+1}^{\ast
}-2y_{n}+y_{n-1}\right\} .
\end{eqnarray*}%
The equivalent linear multistep method is%
\begin{gather}
\frac{1}{1-\frac{\nu }{2}}y_{n+1}-\frac{1+\frac{\nu }{2}}{1-\frac{\nu }{2}}%
y_{n}+\frac{\frac{\nu }{2}}{1-\frac{\nu }{2}}y_{n-1}=  \label{eq:LMMvariable}
\\
=kf(t_{n+1},\frac{1}{1-\frac{\nu }{2}}y_{n+1}-\frac{\nu }{1-\frac{\nu }{2}}%
y_{n}+\frac{\frac{\nu }{2}}{1-\frac{\nu }{2}}y_{n-1}).  \notag
\end{gather}%
This corresponds to%
\begin{equation*}
\begin{array}{ccc}
\alpha _{2}=\frac{1}{1-\frac{\nu }{2}} & , & \beta _{2}=\frac{1}{1-\frac{\nu 
}{2}} \\ 
\\
\alpha _{1}=-\frac{1+\frac{\nu }{2}}{1-\frac{\nu }{2}} & , & \beta _{1}=-%
\frac{\nu }{1-\frac{\nu }{2}} \\ 
\\
\alpha _{0}=\frac{\frac{\nu }{2}}{1-\frac{\nu }{2}} & , & \beta _{0}=\frac{%
\frac{\nu }{2}}{1-\frac{\nu }{2}}.%
\end{array}%
\end{equation*}%
There are various places where $A-$stable $2-$step methods are characterized
in terms of their coefficients, e.g., Dahlquist \cite{D78,D79}, Dahlquist, Liniger and Nevanlinna \cite{DLN83}, 
Grigorieff \cite{G83}, Nevanlinna \cite{N84}. We shall apply the characterization (for variable
timesteps) in Dahlquist \cite{D79}, Lemma 4.1 page 3, 4 (specifically rearranging the
equation on page 4 following (4.1)), which states that the method is $A-$%
stable if%
\begin{equation}
\begin{cases}
-\alpha _{1}\geq 0,\\
\\
1-2\beta _{1}\geq 0\text{ and }\\
\\
2(\beta _{2}\ -\beta
_{0})+\alpha _{1}\geq 0.  \label{eq:AstableConditions}
\end{cases}
\end{equation}

\begin{proposition}
The method (\ref{eq:LMMvariable}) is $0-$stable for 
\begin{equation*}
-2\leq \nu <2
\end{equation*}%
and $0-$unstable otherwise. Let $-2\leq \nu <2$. The method is $A-$stable
for 
\begin{equation*}
-\frac{2}{3}\leq \nu \leq +\frac{2}{3}.
\end{equation*}
\end{proposition}

\begin{proof}
For $0-$stability, the associated polynomial is%
\begin{eqnarray*}
&\frac{1}{1-(\nu /2)}z^{2}-\frac{1+(\nu /2)}{1-(\nu /2)}z+\frac{(\nu /2)}{%
1-(\nu /2)} =0 \\
\\
&\Updownarrow\\
\\
&z^{2}-\left( 1+(\nu /2)\right) z+(\nu /2) =0.
\end{eqnarray*}%
Its roots are 
\begin{equation*}
z_{\pm }=1,\frac{\nu }{2},
\end{equation*}%
from which $0-$stability follows for $-2\leq \nu <2$.

To show $A-$stability we apply the characterization (\ref%
{eq:AstableConditions}). Due to the $0-$stability result, restrict to values 
$-2\leq \nu <2$ for which the denominator $$1-\frac{\nu}{2}>0.$$ The first of the
three conditions is%
\begin{gather*}
-\alpha _{1}\geq 0 \\
\Updownarrow \\
\frac{1+\frac{\nu }{2}}{1-\frac{\nu }{2}}%
\geq 0\\
\Updownarrow \\
\nu \geq -2.
\end{gather*}%
The second is%
\begin{gather*}
1-2\beta _{1}\geq 0 \\
\Updownarrow \\
1-2\left( -\frac{\nu }{1-\frac{\nu }{2}}%
\right) \geq 0 \\
\Updownarrow \\
1+\frac{3\nu }{2}\geq 0\\
\Updownarrow \\
\nu \geq -\frac{2}{3}.
\end{gather*}%
The third condition is%
\begin{gather*}
2(\beta _{2}\ -\beta _{0})+\alpha _{1}\geq 0 \\
\Updownarrow \\
 2\left( \frac{1}{%
1-\frac{\nu }{2}}-\frac{\frac{\nu }{2}}{1-\frac{\nu }{2}}\right) -\frac{1+%
\frac{\nu }{2}}{1-\frac{\nu }{2}}\geq 0 \\
\Updownarrow\\
 2\left( 1-\frac{\nu }{2}\right) -1-\frac{\nu }{2}\geq0\\
 \Updownarrow \\
  \nu \leq \frac{2}{3}.
\end{gather*}
\end{proof}

For the interesting choice $\nu =+2/3$ we have computed the stability region
of the induced $2-$step method by the root locus method and present it next
in Figure \ref{Fig:stabregionbeplusfilter}.

\begin{figure}[!htb]
	\centering
	\includegraphics[width=12cm, height=8cm]{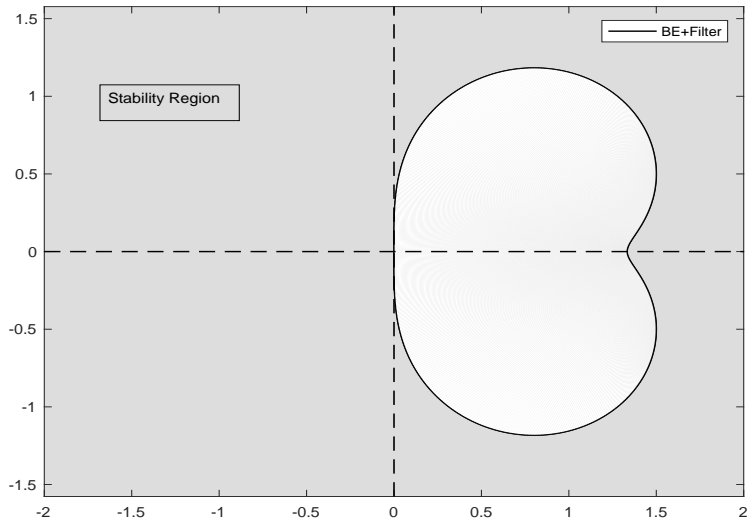}\hspace{-1.em}
	\caption{Stability region of Backward euler plus time filter}
	\label{Fig:stabregionbeplusfilter}
\end{figure}

For comparison, the stability regions of BE and BDF2 follow in Figure \ref{Fig:StabRegion}.%
\begin{figure}[!htb]
	\centering
	\includegraphics[width=6cm, height=6cm]{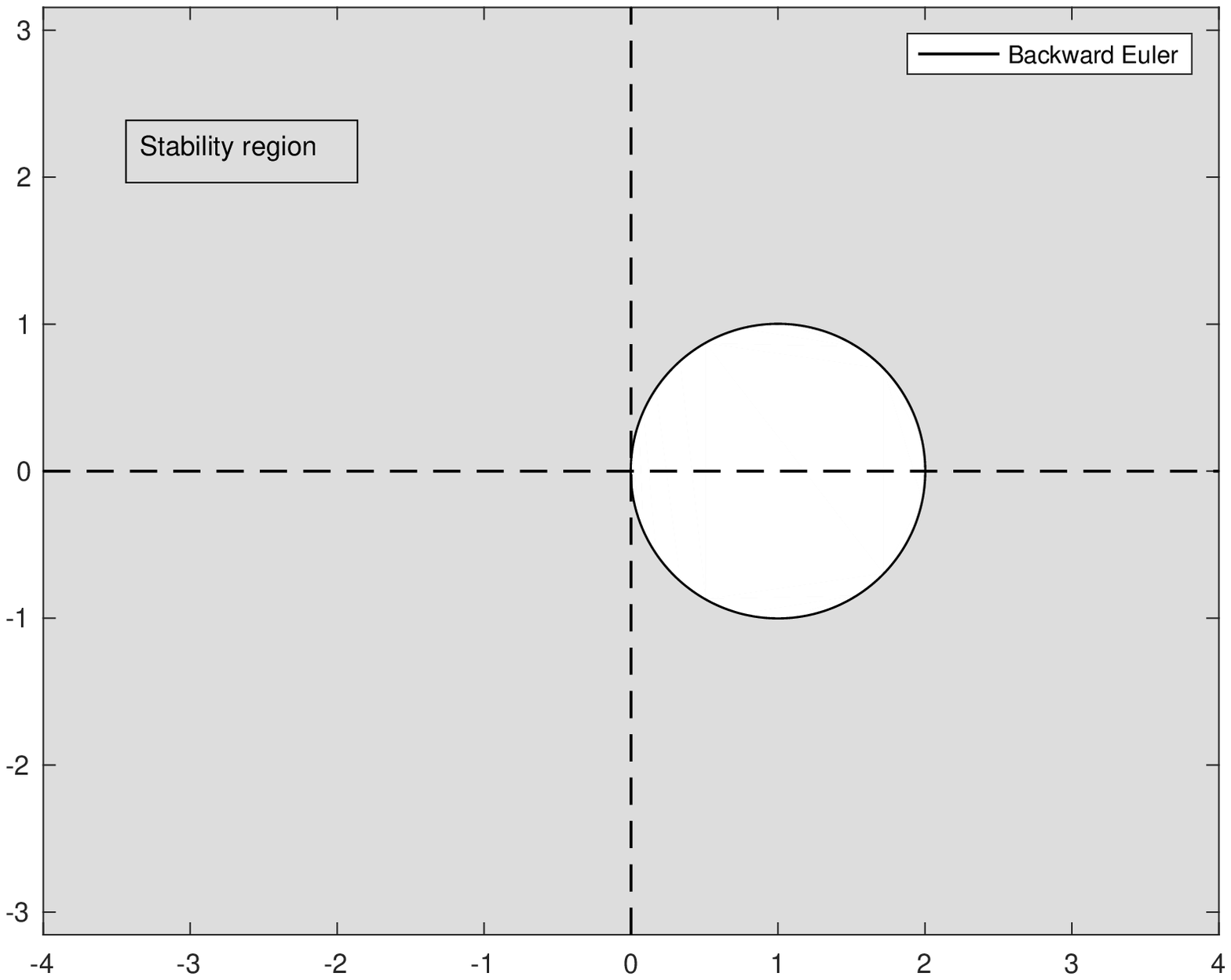} \hspace{-1.em}
	\includegraphics[width=6cm, height=6cm]{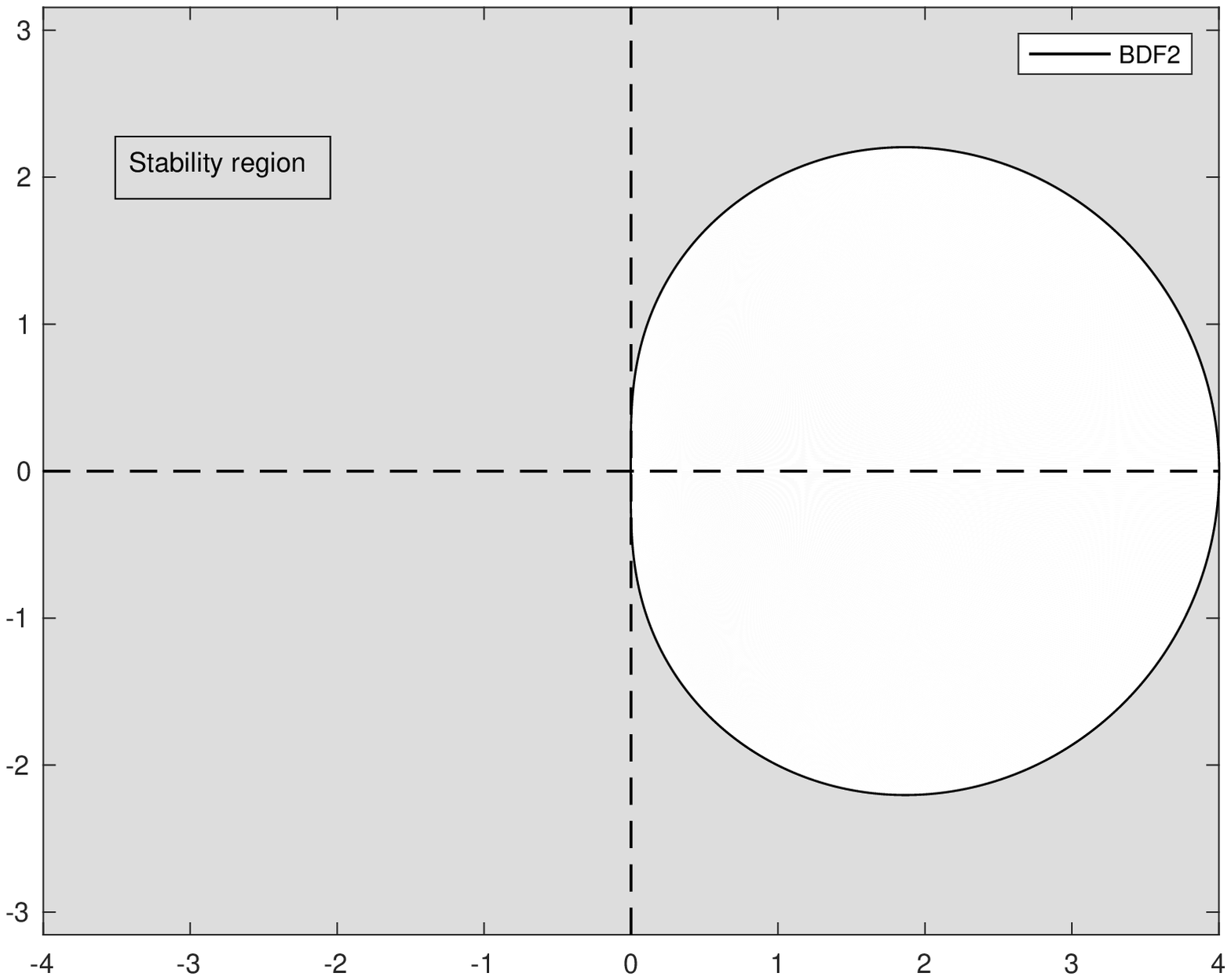} \hspace{-1.em}
	\caption{Stability region of Backward euler(left) and BDF2(right).}
	\label{Fig:StabRegion}
\end{figure}
The boundaries of the three stability regions are presented next in Figure %
\ref{Fig:StabRegionComparisons}.
\begin{figure}[!htb]
	\centering
	\includegraphics[width=12cm, height=8cm]{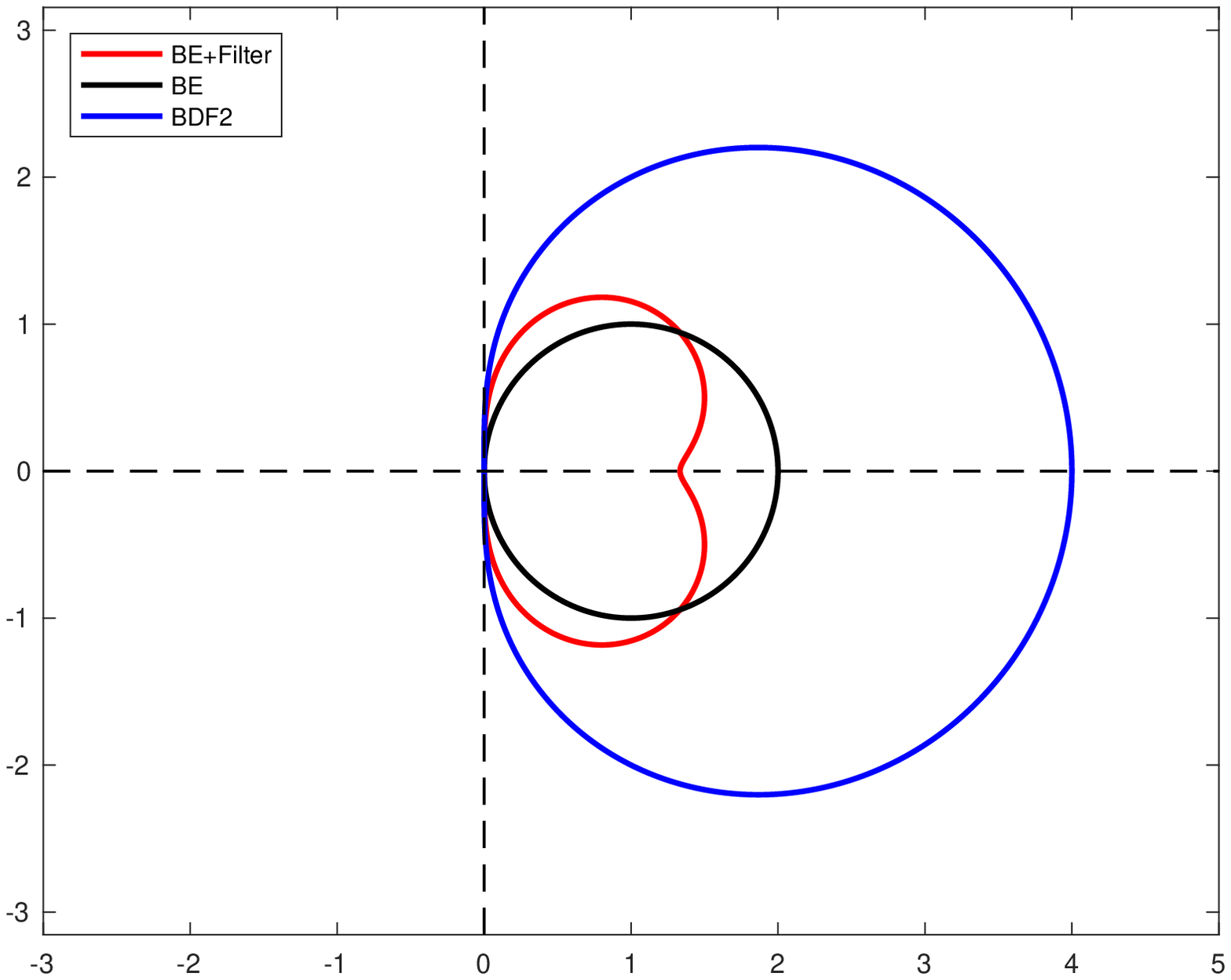} \hspace{-1.em}
	\caption{Boundaries of Stability Regions}
	\label{Fig:StabRegionComparisons}
\end{figure}
\begin{remark}
The stability region of the new method is\ larger than that of BDF2
suggesting the new method is somewhat more dissipative than BDF2. This is
consistent with the numerical results in Section 4.
\end{remark}

\section{Variable Timestep} \hfill \\

Since the implicit method is a one step method the key is to extend time
filters to variable timesteps. We begin.

\subsection{Time Filters on Nonuniform Meshes} 

To extend time filters to nonuniform timesteps we must first define the
discrete curvature. The extension of differential geometry to discrete
settings is an active research fields with considerable work on discrete
curvature, e.g., Najman \cite{N17}. For 3 points the natural definitions are either
the discrete second difference or the inverse of the radius of the
interpolating circle. Consistent with work in GFD, we employ the former
scaled by $k_{n-1}k_{n}$, e.g., Williams \cite{W11}, Kalnay \cite{K03}. Consider the points%
\begin{equation*}
(t_{n-1},y_{n-1}),(t_{n},y_{n}),(t_{n+1},y_{n+1}).
\end{equation*}%
Let the Lagrange basis functions for these three points be denoted $$%
\ell_{n-1}(t),\text{ } \ell_{n}(t),\text{ }  \ell_{n+1}(t).$$ The quadratic interpolant at the
three points is then%
\begin{equation*}
\phi (t)=y_{n+1}\ell_{n+1}(t)+y_{n}\ell_{n}(t)+y_{n-1}\ell_{n-1}(t).
\end{equation*}

\begin{definition}
The discrete curvature at $(t_{n-1},y_{n-1}),(t_{n},y_{n}),(t_{n+1},y_{n+1})$
is 
\begin{eqnarray*}
\kappa _{n} &=&k_{n-1}k_{n}\phi ^{\prime \prime } \\
&=&\frac{2k_{n-1}}{k_{n}+k_{n-1}}y_{n+1}-2y_{n}+\frac{2k_{n}}{k_{n}+k_{n-1}}%
y_{n-1}.
\end{eqnarray*}%
Equivalently, recalling $\tau =\frac{k_{n}}{k_{n-1}}$,%
\begin{equation*}
\kappa _{n}=\frac{2}{1+\tau }y_{n+1}-2y_{n}+\frac{2\tau }{1+\tau }y_{n-1}.
\end{equation*}
\end{definition}

We define the extension of the filter (\ref{eq:FilterConstantStep}) in (\ref%
{eq:methodConstantTimestep}) to nonuniform meshes as%
\begin{equation}
y_{n+1}\Leftarrow y_{n+1}-\frac{\nu }{2}\left\{ \frac{2}{1+\tau }%
y_{n+1}-2y_{n}+\frac{2\tau }{1+\tau }y_{n-1}\right\} .
\label{eq:VariableFilter}
\end{equation}

\begin{proposition}
The filter (\ref{eq:VariableFilter}) alters the discrete curvature before, $%
\kappa ^{old}$, and after, $\kappa ^{new}$, filtering by%
\begin{equation*}
\kappa ^{new}=(1-\frac{\nu }{1+\tau })\kappa ^{old}.
\end{equation*}%
The variable timestep filter reduces, without changing sign, the discrete
curvature, $|\kappa ^{new}|<|\kappa ^{old}|$, provided%
\begin{equation*}
0<\nu <1+\tau .
\end{equation*}
\end{proposition}

\begin{proof}
The first claim follows by algebraic rearrangement of the filter equation (%
\ref{eq:VariableFilter})%
\begin{gather*}
y_{n+1}^{new}=y_{n+1}^{old}-\frac{\nu }{2}\left\{ \frac{2}{1+\tau }%
y_{n+1}^{old}-2y_{n}+\frac{2\tau }{1+\tau }y_{n-1}\right\} ,or \\
\frac{2}{1+\tau }y_{n+1}^{new}=\frac{2}{1+\tau }y_{n+1}^{old}-\frac{\nu }{2}%
\frac{2}{1+\tau }\left\{ \frac{2}{1+\tau }y_{n+1}^{old}-2y_{n}+\frac{2\tau }{%
1+\tau }y_{n-1}\right\} \\
\frac{2}{1+\tau }y_{n+1}^{new}-2y_{n}+\frac{2\tau }{1+\tau }y_{n-1}=\frac{2}{%
1+\tau }y_{n+1}^{old}-2y_{n}+\frac{2\tau }{1+\tau }y_{n-1} \\
-\frac{\nu }{2}\frac{2}{1+\tau }\left\{ \frac{2}{1+\tau }%
y_{n+1}^{old}-2y_{n}+\frac{2\tau }{1+\tau }y_{n-1}\right\} , \\
\kappa ^{new}=(1-\frac{\nu }{1+\tau })\kappa ^{old}.
\end{gather*}%
Curvature reduction thus holds provided 
\begin{equation*}
0<\nu \frac{1}{1+\tau }<1,
\end{equation*}%
as claimed.
\end{proof}

In the next figure the three points $(t_{n-1},y_{n-1})$, $(t_{n},y_{n})$, $%
(t_{n+1},y_{n+1})$ and their quadratic interpolant are depicted. The
discrete curvature is the second\ derivative of the interpolating quadratic
scaled by $k_{n-1}k_{n}$. For $0<\nu <1+\tau $ the filter would move the
value $y_{n+1}$\ down slightly (by $O(k^{2})$) to reduce the curvature.
\begin{figure}[!htb]
	\centering
	\includegraphics[width=12cm, height=8cm]{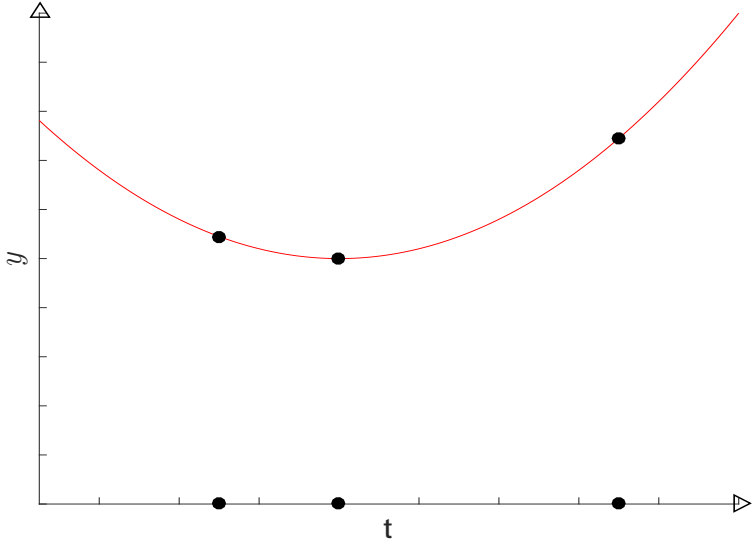}\hspace{-1.em}
	\caption{ $\kappa=k_{n-1}k_{n}\phi''(t)$}
	\label{Fig:curvature}
\end{figure}

\subsection{The local truncation error}

Since the discrete curvature and filter are well defined for variable
timesteps, the method is determined. It is, as presented in (\ref%
{eq:VariableMethod}),%
\begin{equation}
\begin{array}{ccc}
\text{Step 1} & : & \frac{y_{n+1}^{\ast }-y_{n}}{k_{n}}=f(t_{n+1},y_{n+1}^{%
\ast }), \\ 
\\
\text{Step 2} & : & y_{n+1}=y_{n+1}^{\ast }-\frac{\nu }{2}\left\{ \frac{2}{%
1+\tau }y_{n+1}^{\ast }-2y_{n}+\frac{2\tau }{1+\tau }y_{n-1}\right\} .%
\end{array}
\label{eq:VariableStepMethodAgain}
\end{equation}%
Step 2 is used to solve for $y_{n+1}^{\ast }$ and eliminate the prefilter
value by%
\begin{equation*}
y_{n+1}^{\ast }=\frac{1+\tau }{1+\tau -\nu }y_{n+1}-\nu \frac{1+\tau }{%
1+\tau -\nu }y_{n}+\frac{\tau \nu }{1+\tau -\nu }y_{n-1}.
\end{equation*}%
Eliminating $y_{n+1}^{\ast }$\ in Step 1 then gives the equivalent $2-$step
method%
\begin{gather}
\frac{1+\tau }{1+\tau -\nu }y_{n+1}-\nu \frac{1+\tau }{1+\tau -\nu }y_{n}+%
\frac{\tau \nu }{1+\tau -\nu }y_{n-1}-y_{n}=
\label{eq:VariableEquivalentLMM} \\
=k_{n}f(t_{n+1},\frac{1+\tau }{1+\tau -\nu }y_{n+1}-\nu \frac{1+\tau }{%
1+\tau -\nu }y_{n}+\frac{\tau \nu }{1+\tau -\nu }y_{n-1}).  \notag
\end{gather}%
This yields\ the following coefficients%
\begin{equation*}
\begin{array}{ccc}
\alpha _{2}=\frac{1+\tau }{1+\tau -\nu } & , & \beta _{2}=\frac{1+\tau }{%
1+\tau -\nu } \\ 
\\
\alpha _{1}=-\frac{1+\tau +\nu \tau }{1+\tau -\nu } & , & \beta _{1}=-\nu 
\frac{1+\tau }{1+\tau -\nu } \\ 
\\
\alpha _{0}=\frac{\tau \nu }{1+\tau -\nu } & , & \beta _{0}=\frac{\tau \nu }{%
1+\tau -\nu }%
\end{array}%
\end{equation*}%
The $\beta -$coefficients as given above satisfy a standard normalization
condition%
\begin{equation*}
\beta _{2}\ +\beta _{1}\ +\beta _{0}=1.
\end{equation*}%
There is a considerable amount known about $2-$step methods, even with
varying timesteps. Many of the properties of the method follow from applying
the theory in, e.g., Dahlquist \cite{D79}, Dahlquist, Liniger and Nevanlinna \cite{DLN83}, to the above and its variable
timestep analog.

We prove the following.

\begin{proposition}
The variable timestep method (\ref{eq:VariableStepMethodAgain}) is always
consistent. It is second order accurate provided 
\begin{equation*}
\nu =\frac{\tau (1+\tau )}{1+2\tau }.
\end{equation*}%
Moreover, the $LTE$ for $\nu =\frac{\tau (\tau +1)}{1+2\tau }$ is 
\begin{equation*}
LTE=\frac{-(1+4\tau )}{6\tau }k_{n}^{3}y^{\prime \prime \prime }(t_{n})+%
\mathcal{O}(k_{n}^{4}).
\end{equation*}
\end{proposition}

The relation $\nu =\tau (1+\tau )/(1+2\tau )$ for second order accuracy is
plotted below.
\begin{figure}[!htb]
	\centering
	\includegraphics[width=12cm, height=8cm]{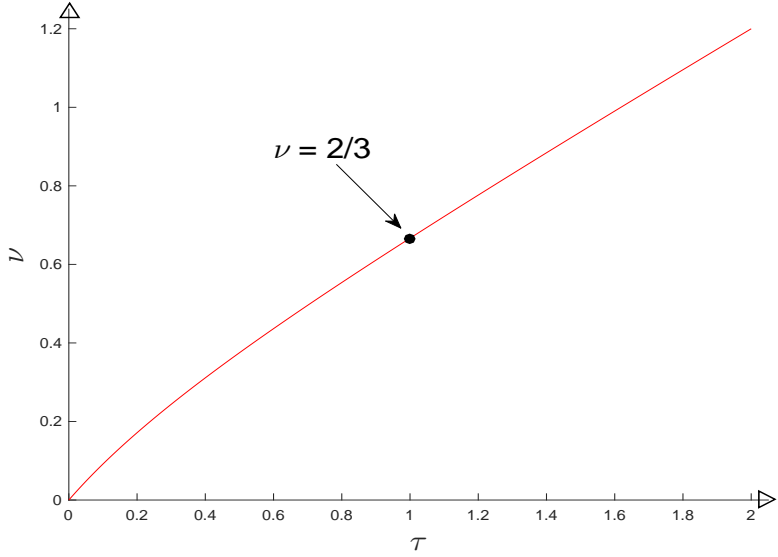} \hspace{-1.em}
	\caption{ Curvature reduction and second order choice of $\nu$}
	\label{Fig:curvaturereduction}
\end{figure}

\begin{proof}
The equivalent $2-$step method corresponds to\ the coefficients%
\begin{equation*}
\begin{array}{ccc}
\alpha _{2}=\frac{1+\tau }{1+\tau -\nu } & , & \beta _{2}=\frac{1+\tau }{%
1+\tau -\nu } \\ 
\\
\alpha _{1}=-\frac{1+\tau +\nu \tau }{1+\tau -\nu } & , & \beta _{1}=-\nu 
\frac{1+\tau }{1+\tau -\nu } \\ 
\\
\alpha _{0}=\frac{\tau \nu }{1+\tau -\nu } & , & \beta _{0}=\frac{\tau \nu }{%
1+\tau -\nu }%
\end{array}%
\end{equation*}%
By a Taylor expansion (the Appendix), the method is consistent if and only if %
the following two condition are satisfied,
\begin{equation*}
\begin{array}{ccc}
Condition \text{ }1: \alpha _{2}+\alpha _{1}+\alpha _{0}=0\\ 
\\
Condition \text{ }2: \alpha _{2}\ -\frac{1}{\tau }\alpha _{0}\ -(\beta _{2}\ +\beta _{1}\ +\beta
_{0}\ )=0
\end{array}%
\end{equation*}%
The first two consistency conditions identically holds. Indeed,%
\begin{gather*}
Condition\text{ }1: \\
\frac{1+\tau }{1+\tau -\nu }+(-\nu \frac{1+\tau }{1+\tau -\nu }-1)+\frac{%
\tau \nu }{1+\tau -\nu }=0 \\
\\
\Updownarrow\\
\\
1+\tau -\nu \lbrack 1+\tau ]-[1+\tau -\nu ]+\tau \nu =0\Leftrightarrow 0=0,
\end{gather*}%
and similarly for Condition 2%
\begin{gather*}
Condition\text{ }2: \\
\frac{1+\tau }{1+\tau -\nu }-\frac{1}{\tau }\frac{\tau \nu }{1+\tau -\nu }-(%
\frac{1+\tau }{1+\tau -\nu }-\nu \frac{1+\tau }{1+\tau -\nu }+\frac{\tau \nu 
}{1+\tau -\nu })=0 \\
\\
\Updownarrow \\
\\
\tau \lbrack 1+\tau ]-\tau \nu -(\tau \lbrack 1+\tau ]-\tau \nu \lbrack
1+\tau ]+\tau \tau \nu )=0\Leftrightarrow 0=0.
\end{gather*}%
Therefore\ the method is always consistent. The method is second order
accurate if and only if%
\begin{gather*}
\frac{1}{2}\alpha _{2}\ +\frac{1}{2\tau ^{2}}\alpha _{0}\ -\beta _{2}\ +%
\frac{1}{\tau }\beta _{0}=0 \\
\Updownarrow \\
\tau ^{2}\alpha _{2}\ +\alpha _{0}\ -2\tau ^{2}\beta _{2}\ +2\tau \beta
_{0}=0 \\
\Updownarrow \\
\tau ^{2}\frac{1+\tau }{1+\tau -\nu }+\frac{\tau \nu }{1+\tau -\nu }-2\tau
^{2}\frac{1+\tau }{1+\tau -\nu }+2\tau \frac{\tau \nu }{1+\tau -\nu }=0 \\
\Updownarrow  \\
\tau ^{2}[1+\tau ]+\tau \nu -2\tau ^{2}[1+\tau ]+2\tau ^{2}\nu =0 \\
\Updownarrow  \\
\tau \lbrack 1+\tau ]+\nu -2\tau \lbrack 1+\tau ]+2\tau \nu =0 \\
\Updownarrow  \\
\tau \lbrack 1+\tau ]+\nu \lbrack 1+2\tau ]-2\tau \lbrack 1+\tau ]=0 \\
\Updownarrow  \\
\nu \lbrack 1+2\tau ]=-\tau \lbrack 1+\tau ]+2\tau \lbrack 1+\tau ]=\tau
+\tau ^{2} \\
\Updownarrow  \\
\nu =\frac{\tau +\tau ^{2}}{1+2\tau },
\end{gather*}%
as claimed. That the $LTE$ for $\nu =\frac{\tau (\tau +1)}{1+2\tau }$ is 
\begin{equation*}
LTE=\frac{-(1+4\tau )}{6\tau }k_{n}^{3}y^{\prime \prime \prime }(t_{n})+%
\mathcal{O}(k_{n}^{4})
\end{equation*}%
is a calculation of the first non-zero term of the $LTE$ expansion.
\end{proof}

\begin{remark}
BDF2 is related to the method herein. The normal, fully variable BDF2 method
is given by%
\begin{equation}
\frac{2\tau +1}{\tau +1}y_{n+1}-(\tau +1)y_{n}+\frac{\tau ^{2}}{\tau +1}%
y_{n-1}=k_{n}f(t_{n+1},y_{n+1}).
\end{equation}%
By comparison, the equivalent, variable step linear multistep method\ herein
is%
\begin{gather*}
\frac{1+\tau }{1+\tau -\nu }y_{n+1}-\nu \frac{1+\tau }{1+\tau -\nu }y_{n}+%
\frac{\tau \nu }{1+\tau -\nu }y_{n-1}-y_{n}= \\
=k_{n}f(t_{n+1},\frac{1+\tau }{1+\tau -\nu }y_{n+1}-\nu \frac{1+\tau }{%
1+\tau -\nu }y_{n}+\frac{\tau \nu }{1+\tau -\nu }y_{n-1}),
\end{gather*}%
For $\nu =\tau (1+\tau )/(1+2\tau )$\ the LHS is again the same as (variable
step) BDF2 while the RHS differs.
\end{remark}

\subsection{Stability for variable step sizes}

As defined by Dahlquist, Liniger and Nevanlinna \cite{DLN83} equation (1.12)
p.1072, a variable step size method is $A-$stable if, when applied as a
one-leg scheme to 
\begin{equation*}
y^{\prime }=\lambda (t)y,\func{Re}(\lambda (t))\leq 0,
\end{equation*}%
solutions are always bounded for any sequence of step sizes and any such $%
\lambda (t)$. We analyze $A-$stability for variable step sizes applying the
same conditions as for constant step sizes since they were derived in Dahlquist \cite%
{D79}, Dahlquist, Liniger and Nevanlinna \cite{DLN83} for variable step, $2-$step methods. Specifically, we
apply the characterization in Dahlquist \cite{D79}, Lemma 4.1 page 3, 4 (specifically
rearranging the equation on page 4 following (4.1)), which states that the
method is $A-$stable if%
\begin{equation}
\begin{cases}
-\alpha _{1}\geq 0,\\
\\
1-2\beta _{1}\geq 0\text{ and }\\
\\
2(\beta _{2}\ -\beta_{0})+\alpha _{1}\geq 0
\\
\end{cases}
\end{equation}%
The coefficients for (\ref{eq:VariableEquivalentLMM}) are 
\begin{equation*}
\begin{array}{ccc}
\alpha _{2}=\frac{1+\tau }{1+\tau -\nu } & , & \beta _{2}=\frac{1+\tau }{%
1+\tau -\nu } \\ 
\\
\alpha _{1}=-\frac{1+\tau +\nu \tau }{1+\tau -\nu } & , & \beta _{1}=-\nu 
\frac{1+\tau }{1+\tau -\nu } \\ 
\\
\alpha _{0}=\frac{\tau \nu }{1+\tau -\nu } & , & \beta _{0}=\frac{\tau \nu }{%
1+\tau -\nu }.%
\end{array}%
\end{equation*}

\begin{proposition}
The method (\ref{eq:LMMvariable}) is $A-$stable for 
\begin{equation*}
-\frac{1+\tau }{1+2\tau }\leq \nu \leq \min \{\frac{1+\tau }{3\tau },1+\tau
\}.
\end{equation*}
\end{proposition}

\begin{proof}
We check the 3 conditions. The first is%
\begin{gather*}
-\alpha _{1}\geq 0 \\
\Updownarrow \\
\frac{1+\tau +\nu \tau }{1+\tau -\nu }\geq
0.
\end{gather*}%
Considering cases this holds if and only if%
\begin{equation*}
-\frac{1+\tau }{\tau }\leq \nu \leq 1+\tau .
\end{equation*}%
The second is%
\begin{gather*}
1-2\beta _{1}\geq 0 \\
\Updownarrow \\
1+2\nu \frac{1+\tau }{1+\tau -\nu }\geq 0
\end{gather*}%
Since Condition 1 requires $\nu \leq 1+\tau $\ a case is eliminated and this
holds provided 
\begin{equation*}
-\frac{1+\tau }{1+2\tau }\leq \nu .
\end{equation*}%
Condition 3 is%
\begin{gather*}
\text{ }2(\beta _{2}\ -\beta _{0})+\alpha _{1}\geq 0 \\
\Updownarrow \\
2\left( \frac{1+\tau }{1+\tau -\nu }-\frac{\tau \nu }{1+\tau -\nu }\right) -%
\frac{1+\tau +\nu \tau }{1+\tau -\nu }\geq 0.
\end{gather*}%
Since Condition 1 requires $\nu \leq 1+\tau $\ this holds if and only if%
\begin{gather*}
1+\tau \geq +3\nu \tau \\
\Updownarrow\\
\nu \leq \frac{1+\tau }{3\tau }.
\end{gather*}%
Since $$\min \{\frac{1+\tau }{\tau },\frac{1+\tau }{1+2\tau }\}=\frac{1+\tau 
}{1+2\tau }$$ the result follows.
\end{proof}

Since the filter is curvature reducing only for $0<\nu <1+\tau $ it is
sensible to restrict the values to $$0<\nu \leq \min \{\frac{1+\tau }{3\tau }%
,1+\tau \}.$$ We plot next the region in Figure \ref{Fig:VariableAstability},
below the dark curve, in the $\left( \tau ,\nu \right) $ plane of variable
step $A-$stability. Also plotted, the dashed curve, is the choice of $\nu
=\nu (\tau )$ that yields second order accuracy. We see that constant or
reducing the timestep ensures $A-$stability while increasing the timestep
one must either accept first order accuracy with $A-$stability or second
order with some reduced (and yet undetermined) $A(\theta )-$stability, $%
\theta <\pi /2$.
\begin{figure}[!htb]
	\centering
	\includegraphics[width=12cm, height=8cm]{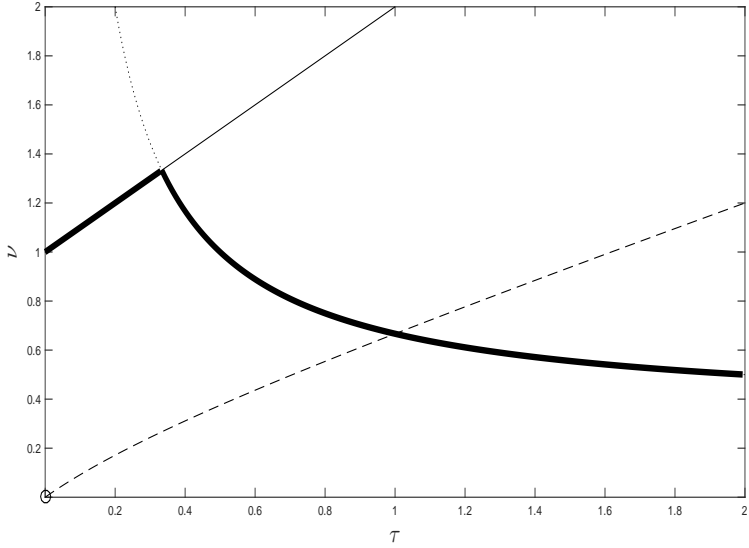} \hspace{-1.em}
	
	\caption{ A-stable for $\nu \leq $ dark\ curve, Dashed Curve = $O(k^{2})$}
	\label{Fig:VariableAstability}
\end{figure}

\begin{remark}
For variable step BDF2 the same conditions can be applied. The result after
some algebra is that the third condition for $A-$stability holds for 
\begin{equation*}
\tau \leq 1.
\end{equation*}%
This is the same constraint that occurs for the method herein when $\nu $ is
restricted to the curve of second order accuracy in Figure \ref%
{Fig:VariableAstability}.
\end{remark}

\subsection{Modified equation analysis}

\hfill \newline
Consider oscillation equation 
\begin{equation}
y^{\prime }(t)=i\omega y(t)\quad \text{and}\quad y(0)=1.
\label{eq:oscillation1}
\end{equation}%
The linear multistep method (\ref{eq:VariableMethod}) is generally a first
order approximation to oscillation equation (\ref{eq:oscillation1}) and
second order for the choice $\nu =\tau (1+\tau )/(1+2\tau )$. To delineate
the distribution of error between phase error and amplitude error we
construct the modified equation of the method for the oscillation equation.
We note that the modified equation is based on an expansion that assumes
implicitly condition $|\omega k_{n}|<1$.

\begin{proposition}
\label{propositionc1c2c3} The three term modified equation of oscillation
equation(\ref{eq:oscillation1}) for (\ref{eq:VariableMethod}) is 
\begin{equation}
\begin{split}
u^{\prime }(t)& =i\omega u(t)+k_{n}C_{1}(i\omega
)^{2}u(t)+k_{n}^{2}C_{2}(i\omega )^{3}u(t)+k_{n}^{3}C_{3}(i\omega )^{4}u(t),
\\
u(0)& =1.
\end{split}
\label{eq:modifiedEqn}
\end{equation}%
where $C_{1}$, $C_{2}$, $C_{3}$, are 
\begin{align*}
C_{1}& =\frac{\tau +\tau ^{2}-\nu -2\nu \tau }{2\tau (1+\tau -\nu )} \\
C_{2}& =\frac{2\tau ^{4}+\tau ^{3}(4-5\nu )+\nu (1+2\nu )+\tau \nu (1+6\nu
)+\tau ^{2}(2-5\nu +6\nu ^{2})}{6\tau ^{2}(1+\tau -\nu )^{2}} \\
C_{3}& =\frac{6\tau ^{6}+\tau ^{5}(18-20\nu )-\tau \nu ^{2}(31+24\nu )+\tau
^{2}\nu (4-33\nu -36\nu ^{2})}{24\tau ^{3}(1+\tau -\nu )^{3}} \\
& +\frac{\tau ^{4}(18-39\nu +23\nu ^{2})+\tau ^{3}(6-16\nu +13\nu ^{2}-24\nu
^{3})-\nu (1+8\nu +6\nu ^{2})}{24\tau ^{3}(1+\tau -\nu )^{3}}.
\end{align*}
\end{proposition}

\begin{proof}
The general three term modified equation of oscillation equation(\ref%
{eq:oscillation1}) takes the form 
\begin{equation*}
u^{\prime }=i\omega u+k_{n}g_{1}(u)+k_{n}^{2}g_{2}(u)+k_{n}^{3}g_{3}(u)\quad %
\mbox{and}\quad u(0)=1.
\end{equation*}%

Thus, 
\begin{align*}
u^{\prime \prime }&=-\omega^2 u+i\omega k_{n}g_{1}(u)+i\omega k_{n}^{2}g_{2}(u) \\
& +i\omega k_{n}g_{1}^{\prime }(u)u+k_{n}^{2}g_{1}^{\prime
}(u)g_{1}(u)+i\omega k_{n}^{2}g_{2}^{\prime }(u)u+\mathcal{O}(k_{n}^{3}) \\
u^{\prime \prime \prime }&=-i\omega^3u-\omega ^{2}k_{n}g_{1}(u)-2\omega
^{2}k_{n}g_{1}^{\prime }(u)u+\mathcal{O}(k_{n}^{2}) \\
u^{(4)}& =w^{4}u+\mathcal{O}%
(k_{n}).
\end{align*}%
Consider (\ref{eq:VariableMethod}) applied to oscillation equation, 
\begin{equation*}
y_{n+1}-\frac{\nu \tau +1+\tau }{\tau +1}y_{n}+\frac{\nu \tau }{1+\tau }%
y_{n-1}=i\omega k_{n}y_{n+1}-i\omega k_{n}\nu y_{n}+\frac{\nu \tau }{1+\tau }%
i\omega k_{n}y_{n-1}.
\end{equation*}%
Rearrange term and eliminate $y_{n+1}$, we obtain 
\begin{equation*}
y_{n+1}=\frac{1}{1-i\omega k_{n}}\left( \frac{\nu \tau +1+\tau }{\tau +1}%
y_{n}-\frac{\nu \tau }{1+\tau }y_{n-1}-i\omega k_{n}\nu y_{n}+\frac{\nu \tau 
}{1+\tau }i\omega k_{n}y_{n-1}\right) .
\end{equation*}%
Since $|\omega k_{n}|<1$, we can use approximation of $\frac{1}{1-i\omega
k_{n}}=1+i\omega k_{n}-\omega ^{2}k_{n}^{2}-i\omega ^{3}k_{n}^{3}+\omega
^{4}k_{n}^{4}+\mathcal{O}(k_{n}^{5})$. Therefore, 
\begin{align*}
y_{n+1}& =\frac{\nu \tau +1+\tau }{\tau +1}y_{n}-\frac{\nu \tau }{1+\tau }%
y_{n-1}+\left( \frac{1+\tau -\nu }{\tau +1}\right) i\omega k_{n}y_{n} \\
& +\left( \frac{\nu -1-\tau }{\tau +1}\right) \omega ^{2}k_{n}^{2}y_{n} \\
& +\left( \frac{\nu -1-\tau }{\tau +1}\right) i\omega
^{3}k_{n}^{3}y_{n}+\left( \frac{1+\tau -\nu }{\tau +1}\right) \omega
^{4}k_{n}^{4}y_{n}+\mathcal{O}(k_{n}^{5})
\end{align*}%
The local truncation error of variable stepsize method (\ref%
{eq:VariableMethod}) with modified equations is 
\begin{align*}
LTE& =u(t_{n+1})-y_{n+1} \\
& =u(t_{n+1})-\frac{\nu \tau +1+\tau }{\tau +1}y_{n}+\frac{\nu \tau }{1+\tau 
}y_{n-1}-\left( \frac{1+\tau -\nu }{\tau +1}\right) i\omega k_{n}y_{n} \\
& -\left( \frac{\nu -1-\tau }{\tau +1}\right) \omega
^{2}k_{n}^{2}y_{n}-\left( \frac{\nu -1-\tau }{\tau +1}\right) i\omega
^{3}k_{n}^{3}y_{n} \\
& -\left( \frac{1+\tau -\nu }{\tau +1}\right) \omega ^{4}k_{n}^{4}y_{n}+%
\mathcal{O}(k_{n}^{5})
\end{align*}%
Assume that numerical solution of all previous time steps are exact i.e. $%
y_{i}=u(t_{i})$ for all $i=1\cdots n$ , 
\begin{align*}
LTE& =u(t_{n+1})-\frac{\nu \tau +1+\tau }{\tau +1}u(t_{n})+\frac{\nu \tau }{%
1+\tau }u(t_{n-1})-\left( \frac{1+\tau -\nu }{\tau +1}\right) i\omega
k_{n}u(t_{n}) \\
& -\left( \frac{\nu -1-\tau }{\tau +1}\right) \omega
^{2}k_{n}^{2}u(t_{n})-\left( \frac{\nu -1-\tau }{\tau +1}\right) i\omega
^{3}k_{n}^{3}u(t_{n}) \\
& -\left( \frac{1+\tau -\nu }{\tau +1}\right) \omega ^{4}k_{n}^{4}u(t_{n})+%
\mathcal{O}(k_{n}^{5}).
\end{align*}%
Apply the Taylor expansion of $u(t_{n-1})$, $u(t_{n+1})$ at time $t_{n}$ and
substitute $u(t_{n+1})$ ,$u(t_{n-1})$,$u^{\prime }(t_{n})$, $u^{\prime
\prime }(t_{n})$, $u^{\prime \prime \prime }(t_{n})$ and $u^{(4)}(t_{n})$ in 
$LTE$, we get 
\begin{gather*}
LTE= \\
\left[ \left( \frac{1+\tau -\nu }{1+\tau }\right) g_{1}(u(t_{n}))-\frac{1}{2}%
\omega ^{2}u(t_{n})-\frac{\nu }{2(\tau +\tau ^{2})}\omega
^{2}u(t_{n})-\left( \frac{\nu -1-\tau }{\tau +1}\right) \omega ^{2}u(t_{n})%
\right] k_{n}^{2} \\
+\bigg[\bigg(\frac{1}{2}+\frac{\nu }{2(\tau +\tau ^{2})}\bigg)i\omega
g_{1}(u(t_{n}))+\bigg(\frac{1}{2}+\frac{\nu }{2(\tau +\tau ^{2})}\bigg)%
i\omega g_{1}^{\prime }(u(t_{n}))u(t_{n}) \\
-\bigg(\frac{1}{6}-\frac{\nu }{6(\tau ^{2}+\tau ^{3})}\bigg)i\omega
^{3}u(t_{n})-\bigg(\frac{\nu -1-\tau }{\tau +1}\bigg)i\omega ^{3}u(t_{n})+%
\frac{1+\tau -\nu }{1+\tau }g_{2}(u(t_{n}))\bigg]k_{n}^{3} \\
+\bigg[\frac{1+\tau -\nu }{1+\tau }g_{3}(u(t_{n}))+\bigg(\frac{1}{2}+\frac{%
\nu }{2(\tau +\tau ^{2})}\bigg)i\omega g_{2}(u(t_{n})) \\
+\bigg(\frac{1}{2}+\frac{\nu }{2(\tau +\tau ^{2})}\bigg)%
g_{1}(u(t_{n}))g_{1}^{\prime }(u(t_{n}))+\bigg(\frac{1}{2}+\frac{\nu }{%
2(\tau +\tau ^{2})}\bigg)i\omega g_{2}^{\prime }(u(t_{n}))u(t_{n}) \\
-\bigg(\frac{1}{6}-\frac{\nu }{6(\tau ^{2}+\tau ^{3})}\bigg)\omega
^{2}g_{1}(u(t_{n}))-\bigg(\frac{1}{6}-\frac{\nu }{6(\tau ^{2}+\tau ^{3})}%
\bigg)2\omega ^{2}g_{1}^{\prime }(u(t_{n}))u(t_{n}) \\
+\bigg(\frac{1}{24}+\frac{\nu }{24(\tau ^{3}+\tau ^{4})}\bigg)w^{4}u(t_{n})-%
\bigg(\frac{1+\tau -\nu }{\tau +1}\bigg)\omega ^{4}u(t_{n})\bigg]k_{n}^{4}
\end{gather*}%
Setting coefficient of $k_{n}^{2}$ term equal to zero to find $g_{1}(u)$ 
\begin{equation*}
g_{1}(u)=\frac{2\nu \tau -\tau -\tau ^{2}+\nu }{2\tau (1+\tau -\nu )}\omega
^{2}u=C_{1}(i\omega )^{2}u.
\end{equation*}%
We use $g_{1}(u)$ and set coefficient of $k_{n}^{3}$ equal to zero, we
obtain $g_{2}(u)$ as following, 
\begin{align*}
g_{2}(u)& =-\frac{2\tau ^{4}+\tau ^{3}(4-5\nu )+\nu (1+2\nu )+\tau \nu
(1+6\nu )+\tau ^{2}(2-5\nu +6\nu ^{2})}{6\tau ^{2}(1+\tau -\nu )^{2}}i\omega
^{3}u \\
& =C_{2}(i\omega )^{3}u.
\end{align*}%
Finally, we use $g_{1}(u)$ and $g_{2}(u)$ and set coefficient of $k_{n}^{4}$
to zero, we get 
\begin{align*}
g_{3}(u)& =\bigg[\frac{6\tau ^{6}+\tau ^{5}(18-20\nu )-\tau \nu
^{2}(31+24\nu )+\tau ^{2}\nu (4-33\nu -36\nu ^{2})}{24\tau ^{3}(1+\tau -\nu
)^{3}} \\
& +\frac{\tau ^{4}(18-39\nu +23\nu ^{2})+\tau ^{3}(6-16\nu +13\nu ^{2}-24\nu
^{3})-\nu (1+8\nu +6\nu ^{2})}{24\tau ^{3}(1+\tau -\nu )^{3}}\bigg]\omega
^{4}u \\
& =C_{4}(i\omega )^{4}u.
\end{align*}
\end{proof}

\begin{remark}
The variable stepsize method (\ref{eq:VariableMethod}) is generally a first
order approximation oscillation equation (\ref{eq:oscillation1}) and fourth
order approximation to modified equation (\ref{eq:modifiedEqn}).
\end{remark}

\subsection{The phase and amplitude error}

We use modified equation to analyze phase and amplitude error. Let denote $%
e_{n}$ as error, then 
\begin{align*}
e_{n}& =y(t_{n})-y_{n} \\
& =y(t_{n})-u(t_{n})+u(t_{n})-y_{n}.
\end{align*}%
Since $u(t_{n})-y_{n}$ has fourth order approximation, then $%
y(t_{n})-u(t_{n})$ gives the leading order error generated by variable
stepsize method (\ref{eq:VariableMethod}) (see Durran \cite{D13}).

\begin{theorem}
The phase and amplitude error of variable stepsize method (\ref%
{eq:VariableMethod}) is 
\begin{equation*}
\begin{split}
R-1& =-C_{2}(\omega k_{n})^{2}+\mathcal{O}((\omega k_{n})^{4}) \\
|A|-1& =-C_{1}(\omega k_{n})^{2}+C_{3}(\omega k_{n})^{4}+\mathcal{O}((\omega
k_{n})^{6}).
\end{split}%
\end{equation*}%
where $C_{1},C_{2},C_{3}$ are defined in (\ref{propositionc1c2c3}).
\end{theorem}

\begin{proof}
Consider exact solution of oscillation equation (\ref{eq:oscillation1}) and
modified equation (\ref{eq:modifiedEqn}), 
\begin{align*}
y(t)& =e^{i\omega t}=\cos (\omega t)+i\sin (\omega t) \\
u(t)& =e^{i\omega t+k_{n}C_{1}(i\omega )^{2}t+k_{n}^{2}C_{2}(i\omega
)^{3}t+k_{n}^{3}C_{3}(i\omega )^{4}t} \\
& =e^{-k_{n}C_{1}\omega ^{2}t+k_{n}^{3}C_{3}\omega ^{4}t}\left[ \cos (\omega
t-k_{n}^{2}C_{2}\omega ^{3}t)+i\sin (\omega t-k_{n}^{2}C_{2}\omega ^{3}t)%
\right]
\end{align*}%
Thus, phase error is 
\begin{equation*}
R-1=\frac{arg(u(t))}{arg(y(t))}-1=\frac{\omega t-k_{n}^{2}C_{2}\omega ^{3}t}{%
\omega t}-1=-C_{2}(\omega k_{n})^{2}.
\end{equation*}%
and amplitude error is 
\begin{equation*}
|A|-1=e^{-k_{n}C_{1}\omega ^{2}t+k_{n}^{3}C_{3}\omega ^{4}t}-1.
\end{equation*}%
take $t=k_{n}$ since we are looking for local error and use 
\begin{equation*}
e^{-k_{n}^{2}C_{1}\omega ^{2}+k_{n}^{4}C_{3}\omega ^{4}}\approx
1-C_{1}(\omega k_{n})^{2}+C_{3}(\omega k_{n})^{4}
\end{equation*}%
Thus, 
\begin{equation*}
|A|-1=-C_{1}(\omega k_{n})^{2}+C_{3}(\omega k_{n})^{4}.
\end{equation*}
\end{proof}

\begin{remark}
The phase and amplitude error of backward Euler method is recovered and
consistent with Durran's book result when $\nu =0$ and $\tau =1$.
\end{remark}

\begin{remark}
The variable stepsize method (\ref{eq:VariableMethod}) has second order
accuracy and fourth order amplitude error when $C_{1}=0$ i.e. $\nu =\frac{%
\tau (\tau +1)}{2\tau +1}$.
\end{remark}

\section{Some numerical tests} \hfill \\

We give a few numerical illustrations next. The first tests is for the
Lorenz system and the results are compared with backward Euler (step 1
without step 2) and BDF2. The self-adaptive RKF4-5 solution is taken as the
benchmark solution. The second tests are for a linear and nonlinear exactly
conservative systems. The third test is from Sussman \cite{S10}. His example
is one for which fully, nonlinearly implicit backward Euler preserves
Lyapunov stability of the steady state while the (commonly used in CFD)
linearly implicit method does not. This property is one reason for the fully
implicit method being used in complex applications. We test if adding Step 2
preserves this property.

\subsection{The Lorenz system}

Consider the Lorenz system: 
\begin{equation*}
\begin{array}{l}
\frac{dX}{dt}=\sigma (Y-X),\vspace{0.2cm} \\ 
\frac{dY}{dt}=-XZ+rX-Y,\vspace{0.2cm} \\ 
\frac{dY}{dt}=XY-bZ.%
\end{array}%
\end{equation*}%
The test chooses parameter values from Durran \cite{D91} ,
\begin{align*}
\sigma =12,r=12,b=6\text{ and }\\
(X_{0},Y_{0},Z_{0})=(-10,-10,25).
\end{align*}%
The system is solved over the time interval $[0,5]$ with Backward Euler,
Backward Euler plus filter and BDF2 with constant timestep. A reference
solution is obtained by adaptive RK4-5. We present solutions of the Lorenz
system for several time steps in Figure \ref{Fig:Lorenzsystem}.

\begin{figure}[!htb]
	\centering
	\includegraphics[width=0.5\textwidth]{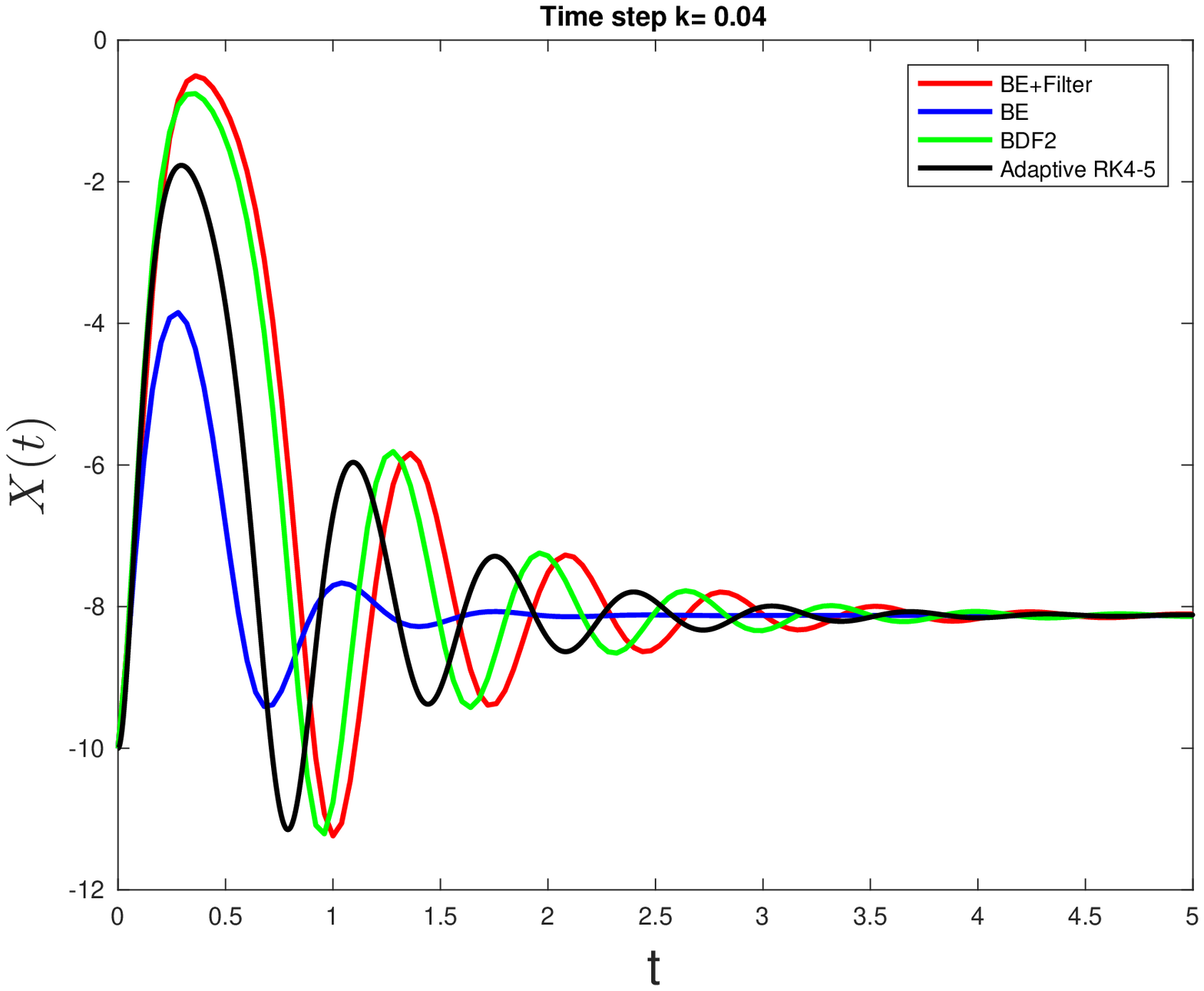} \hspace{-1.em}
	\includegraphics[width=0.5\textwidth]{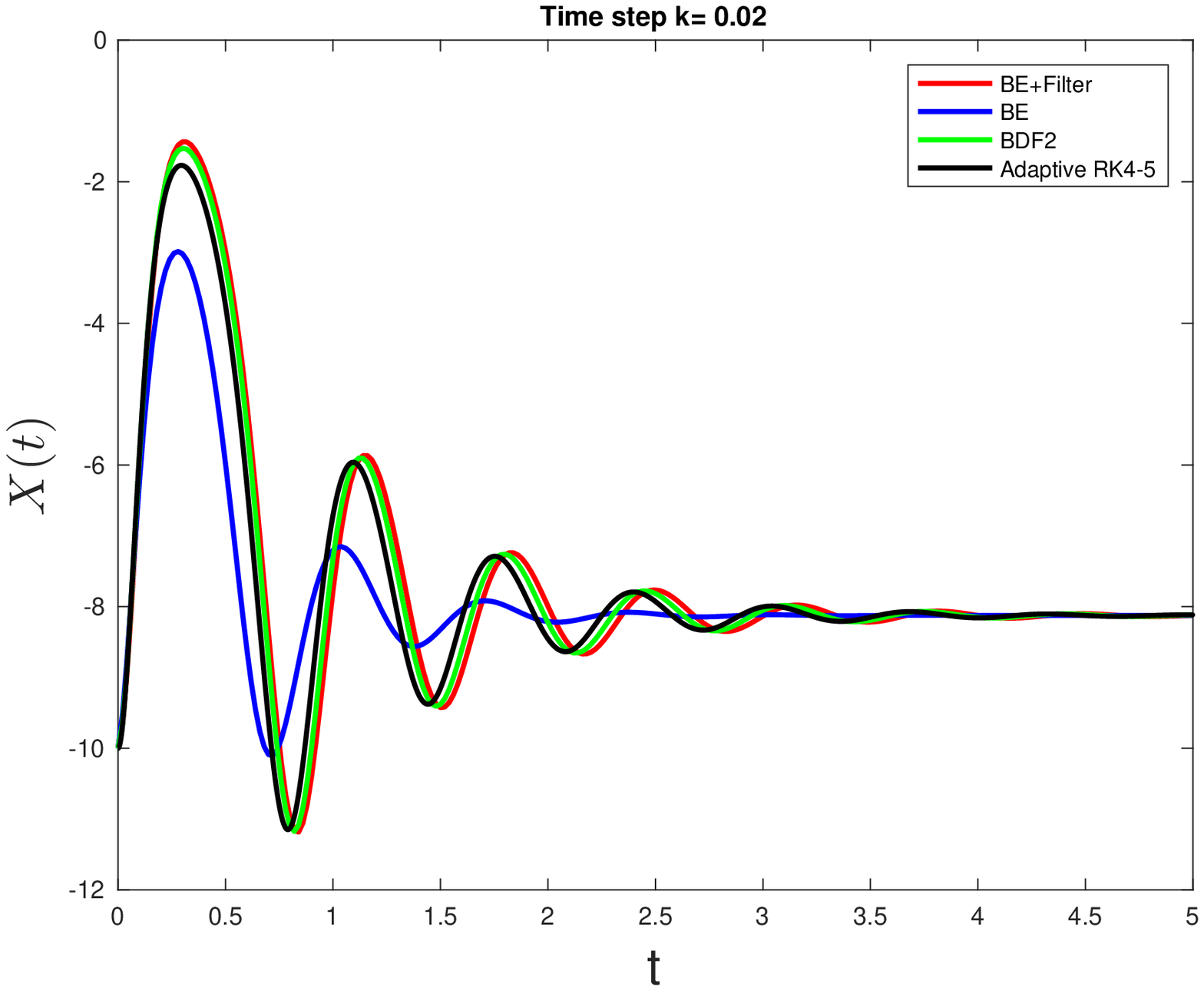} \hspace{-1.em}

	\includegraphics[width=0.5\textwidth]{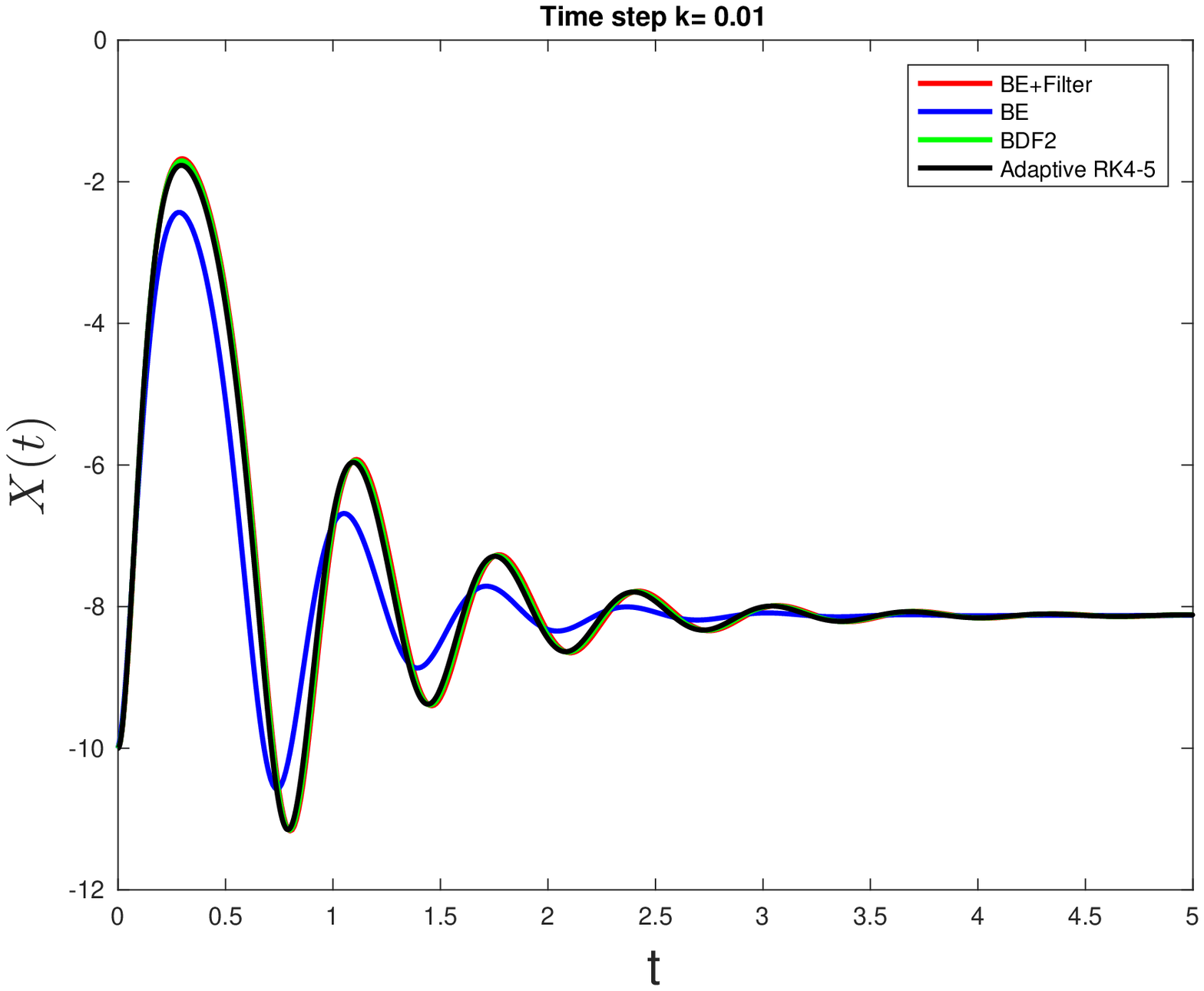}\hspace{-1.em}
	\includegraphics[width=0.5\textwidth]{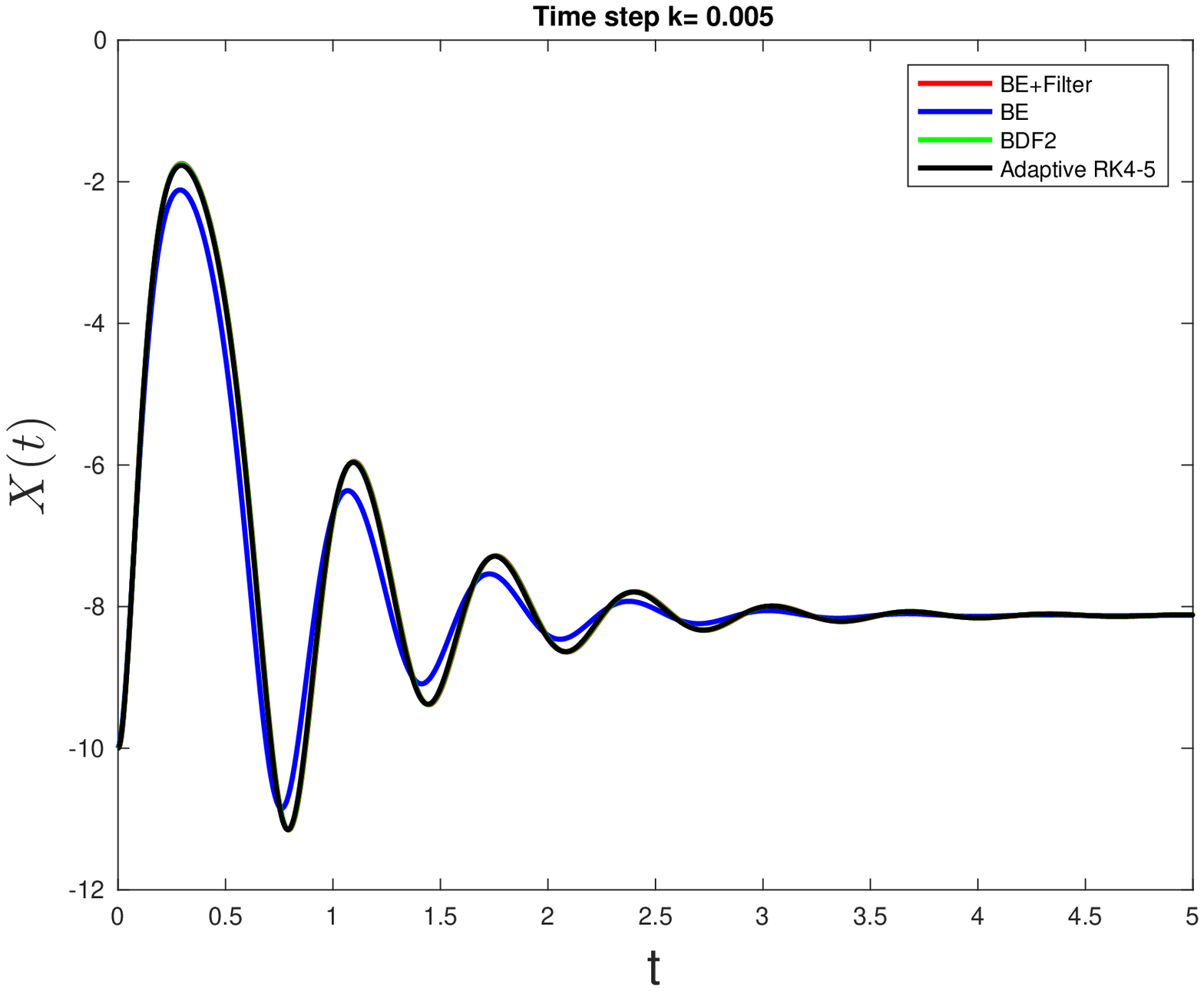}\hspace{-1.em}
	\caption{Lorenz system}
	\label{Fig:Lorenzsystem}
\end{figure}

For moderately small time steps BE over damps severely while both BDF2 and
BE+filter are accurate, even for constant timesteps. Both have a small phase
error that accelerates waves slightly. For large enough time steps, all are\
inaccurate in different ways.

\subsection{Periodic and quasi-periodic oscillations} \hfill \\
\subsubsection{Periodic oscillations}
Consider a simple pendulum problem test problem from Li and Trenchea \cite{LT15}, Williams \cite{W13}
given by 
\begin{align*}
\frac{d\theta }{dt}&=\frac{v}{L}\text{  and }\\
\frac{dv}{dt}&=-g\sin \theta ,
\end{align*}%
where $\theta ,v,L$ and $g$ denote, respectively, angular displacement,
velocity along the arc, length of the pendulum, and the acceleration due to
gravity. Set $$\theta (0)=0.9\pi, \text{ } v(0)=0, \text{ }g=9.8 \text{ and } L=49$$ to observe the long-time behavior of the
numerical solutions in Figure \ref{Fig:SimplePendulum}.

\begin{figure}[!htb]
	\centering
	\includegraphics[width=0.5\textwidth]{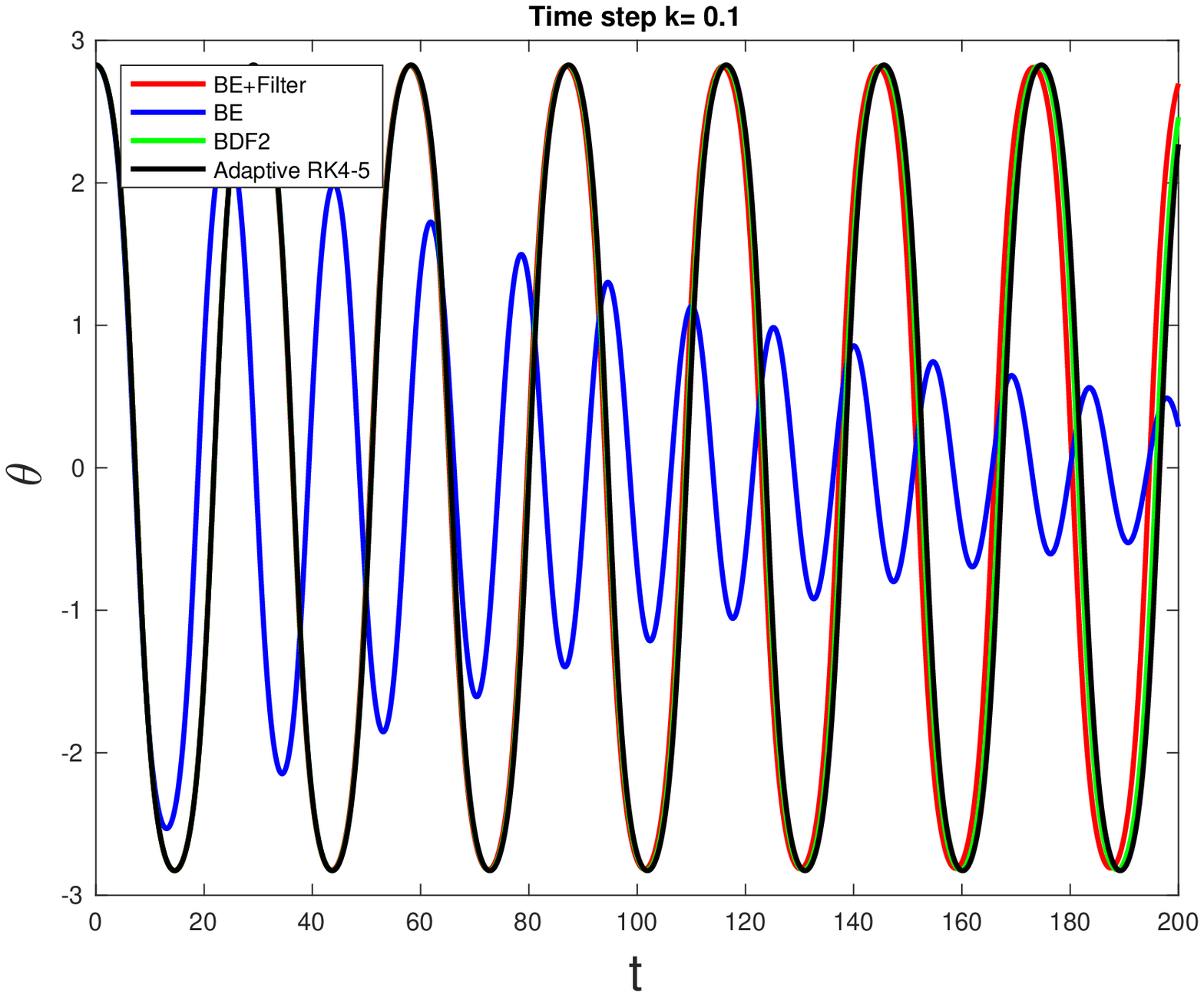} \hspace{-1.em}
	\includegraphics[width=0.5\textwidth]{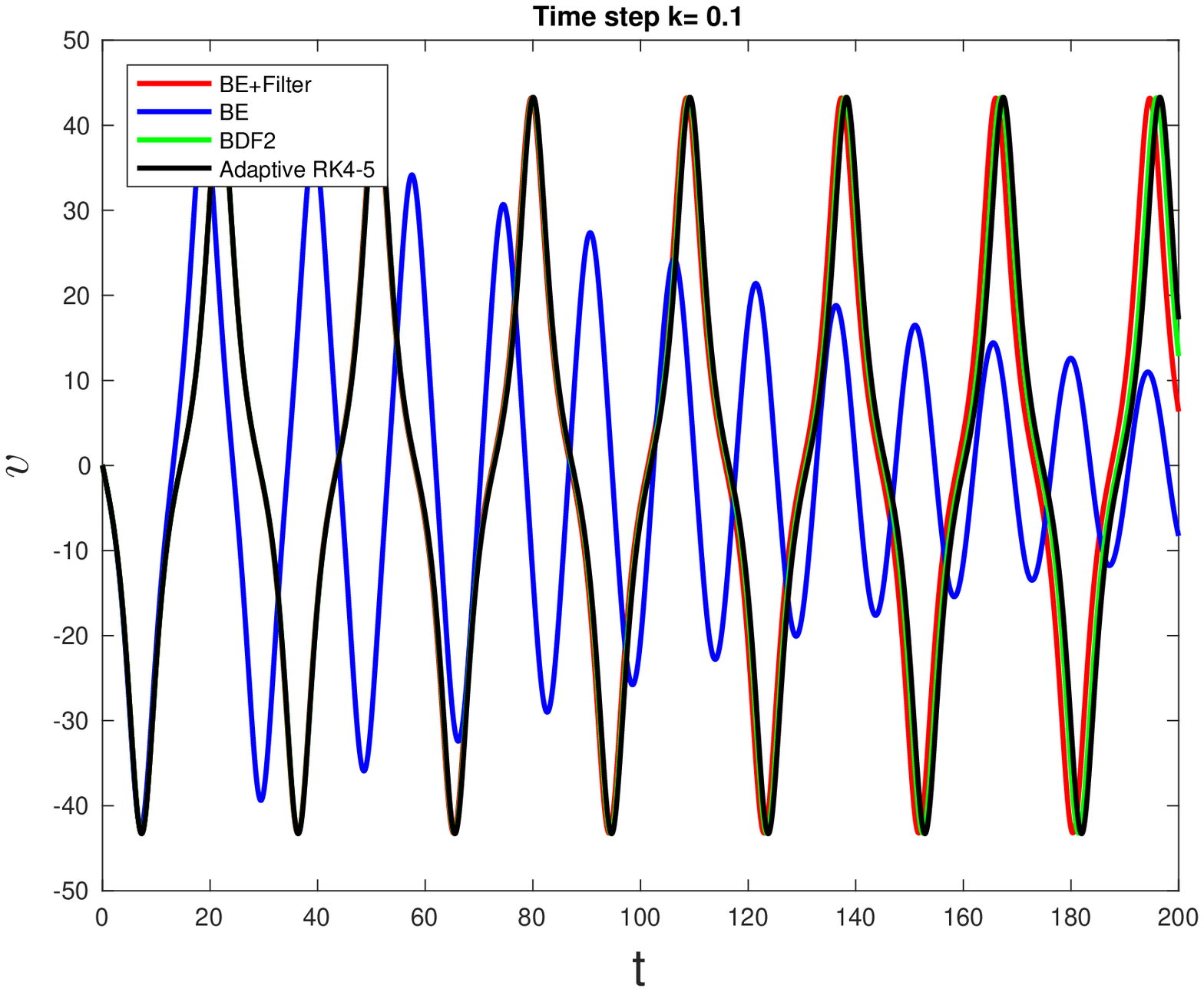} \hspace{-1.em}
	
	\caption{Simple pendulum}
	\label{Fig:SimplePendulum}
\end{figure}

Consistently with test 1, the phase and amplitude errors in both BE+filter
and BDF2 are small while both are large for BE. Adding the filter step to BE
has greatly increased accuracy.

\subsubsection{Quasi-periodic oscillations}
We solve the IVP written as a first order system

\begin{gather*}
x^{\prime \prime \prime \prime }+(\pi ^{2}+1)x^{\prime \prime }+\pi
^{2}x=0,0<t<20, \\
x(0)=2,x^{\prime }(0)=0,x^{\prime \prime }(0)=-(1+\pi ^{2}),x^{\prime \prime
\prime }(0)=0.
\end{gather*}%
This has exact solution $x(t)=cos(t)+cos(\pi t)$, the sum of two periodic
functions with incommensurable periods, hence quasi-periodic, Corduneanu \cite{C89}. We
solve using BE+filter with fixed timestep $k=0.1$ and with a rudimentary
adaptive BE+Filter method. In the latter we use initial timestep $k=0.1$,
the heuristic estimator (\ref{Eq:Estimator}), tolerance $TOL=0.1$, $0.4$ and adapt
by timestep halving and doubling. The plots of both with the exact solution
are next in Figure \ref{Fig:QuasiPeriodic}.
\begin{figure}[!htb]
	\centering
	\includegraphics[width=6cm, height=6cm]{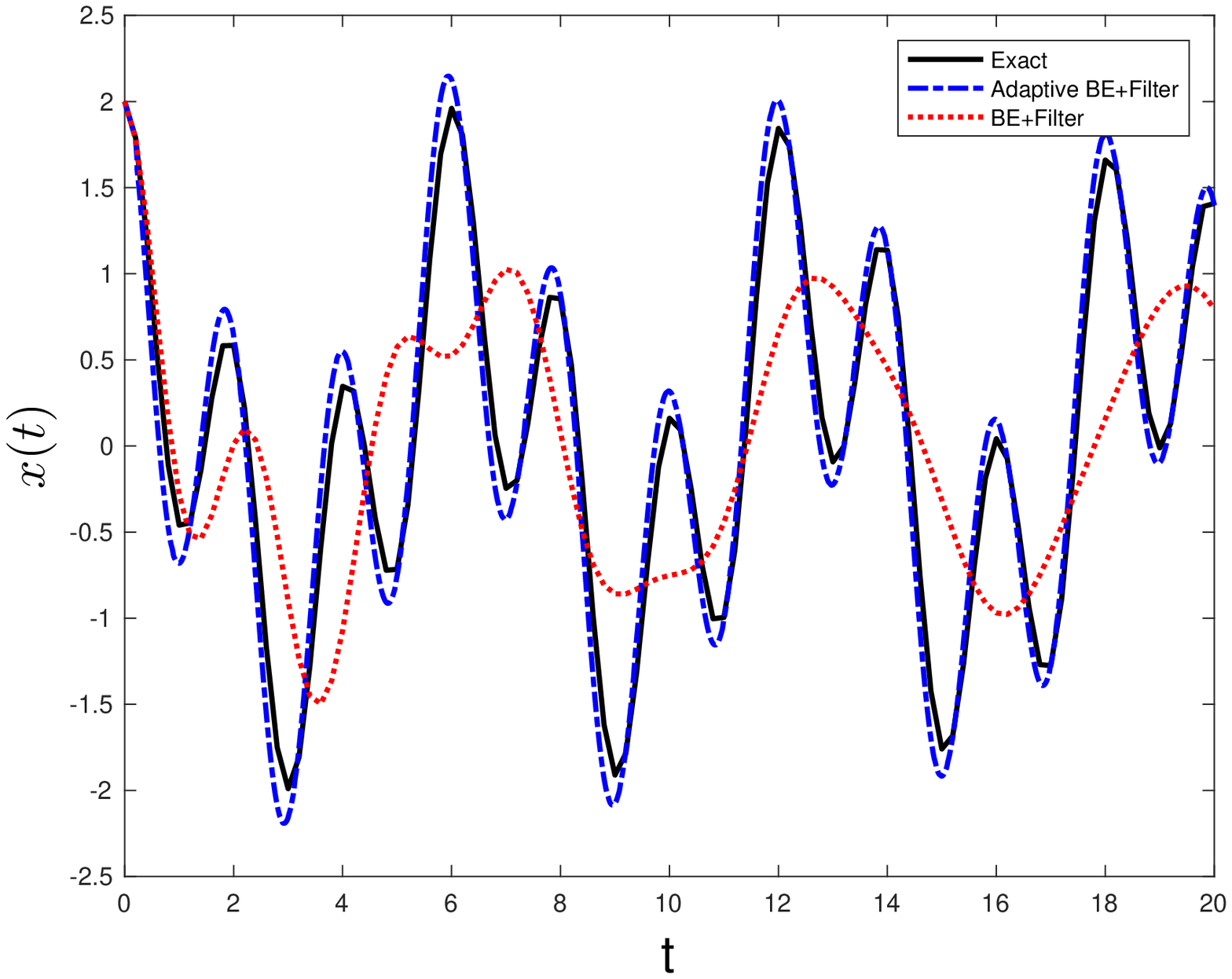} \hspace{-1.em}
	\includegraphics[width=6cm, height=6cm]{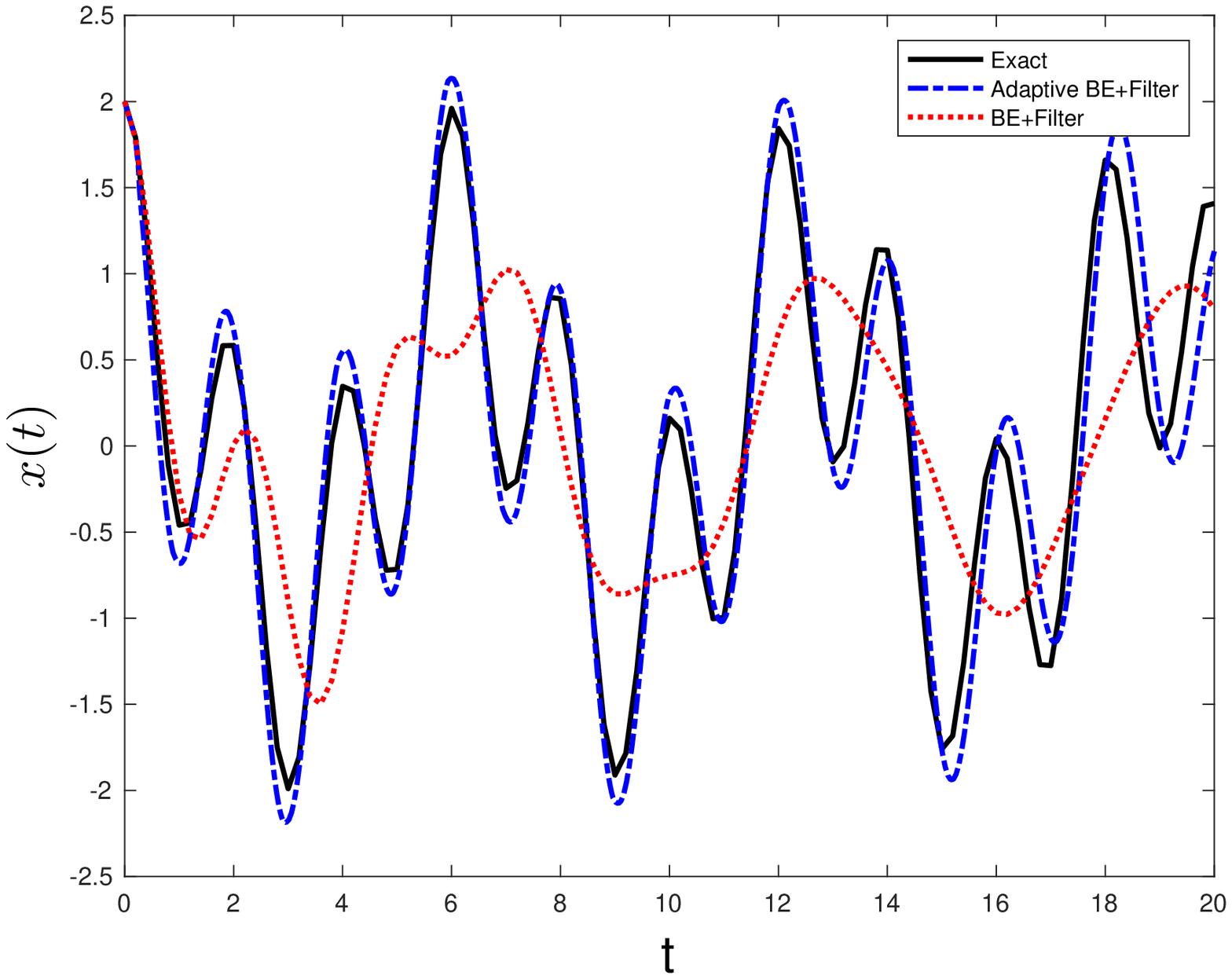}\hspace{-1.em}
	
	\caption{ Quasi-periodic oscillations with TOL = $0.1$(left) and $0.4$(right).}
	\label{Fig:QuasiPeriodic}
\end{figure}

This test suggests that quasi-periodic oscillations are a more challenging test than
periodic. Adaptivity is required but even simple adaptivity suffices to
obtain an accurate solution.

\subsection{The example of Sussman}
Next we present solutions to the test problem of Sussman [S10]. %
He pointed out that the fully, nonlinearly implicit backward Euler %
method approaches steady state while the linearly implicit only %
does so for sufficiently small timestep. In all cases the approximate %
solution approaches steady state as does the behavior of the true %
solution. The nonlinear system is

\begin{equation*}
\begin{array}{l}
\frac{du_{1}}{dt}+u_{2}u_{2}+u_{1}=1,\vspace{0.2cm} \\ 
\frac{du_{2}}{dt}-u_{2}u_{1}+u_{2}=1,%
\end{array}%
\end{equation*}%
with initial value $(u_{1},u_{2})=(0,0)$.%
\begin{figure}[!htb]
	\centering
	\includegraphics[width=6cm, height=6cm]{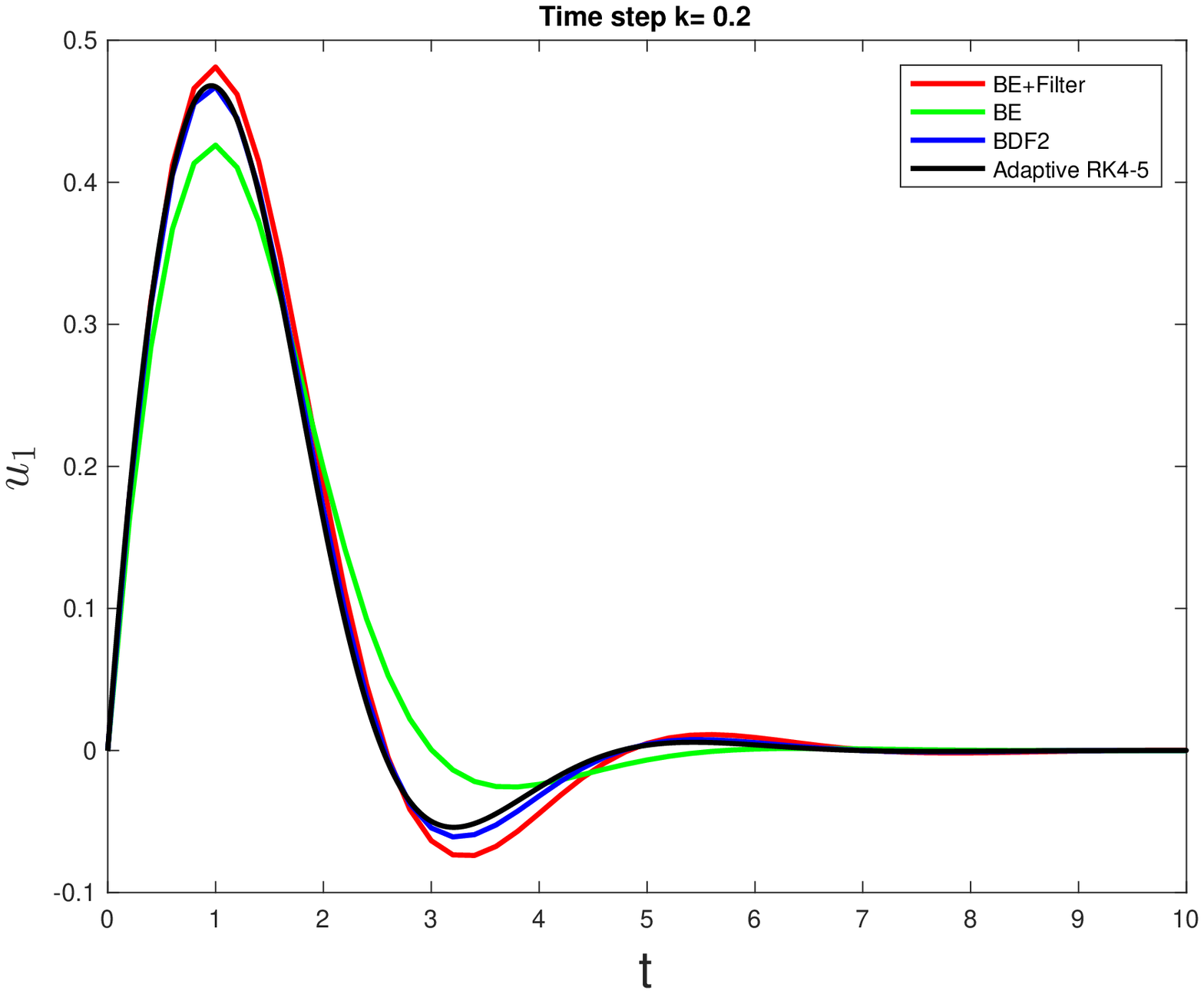} \hspace{-1.em}
	\includegraphics[width=6cm, height=6cm]{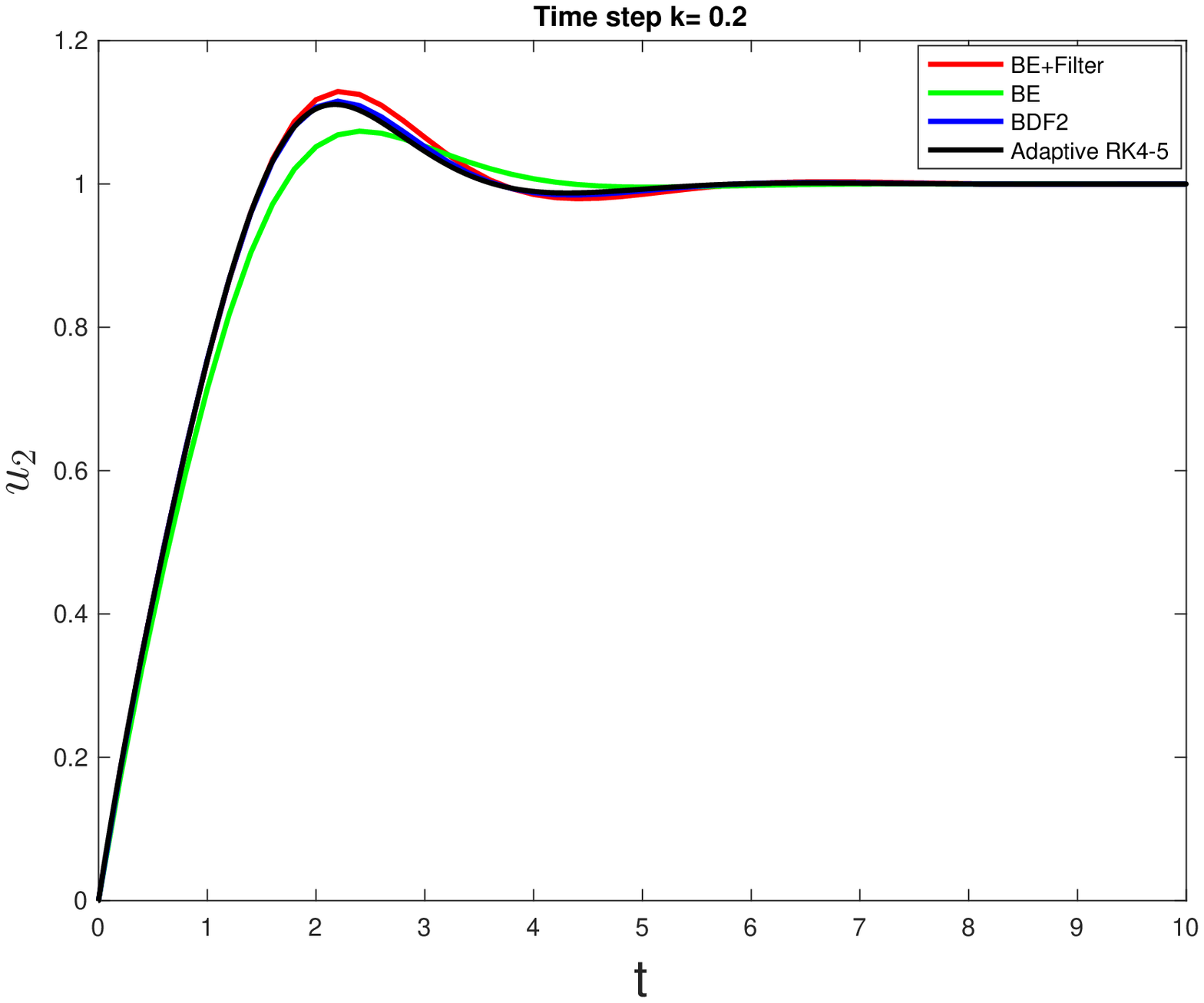} \hspace{-1.em}
	
	\includegraphics[width=6cm, height=6cm]{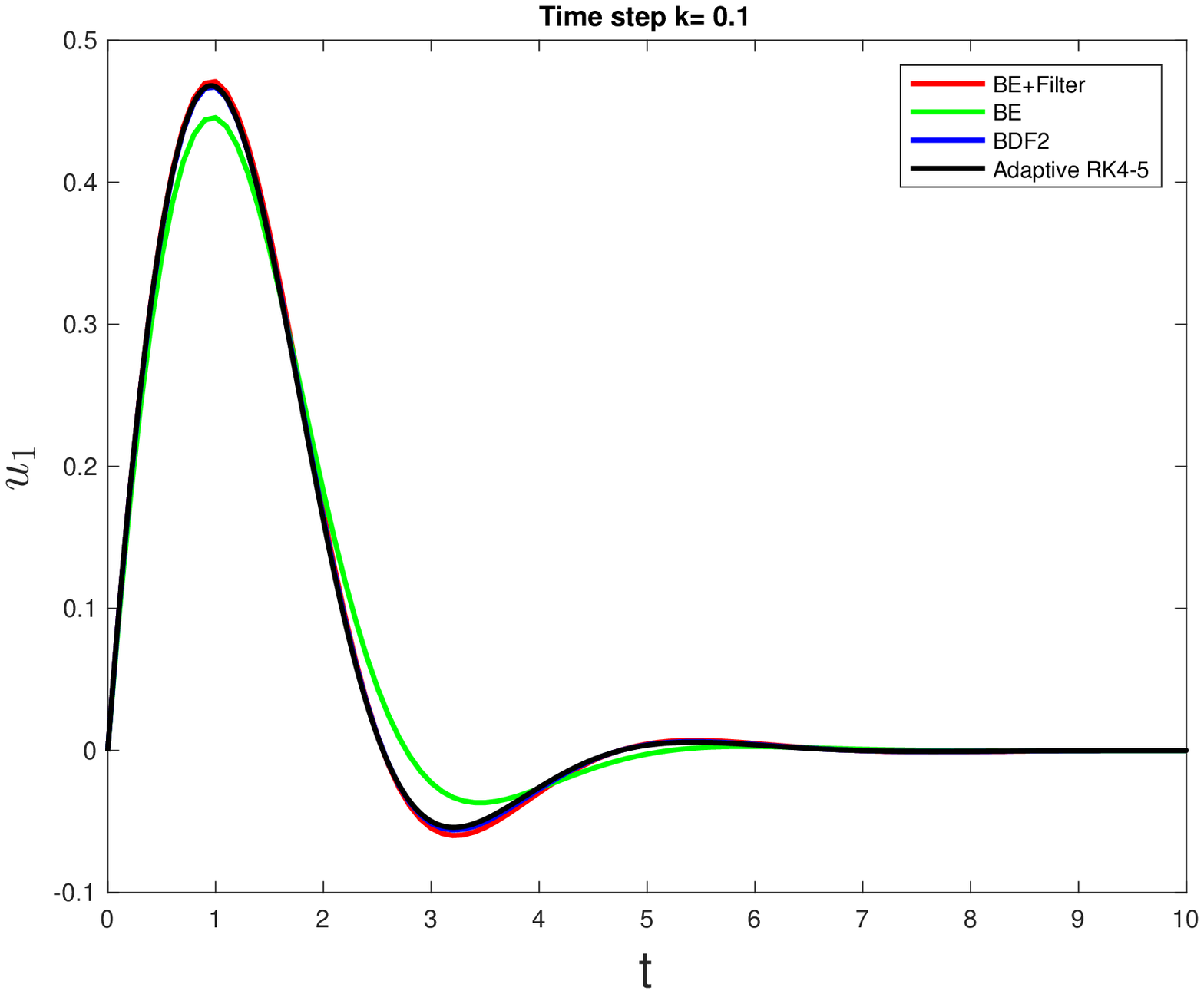} \hspace{-1.em}
	\includegraphics[width=6cm, height=6cm]{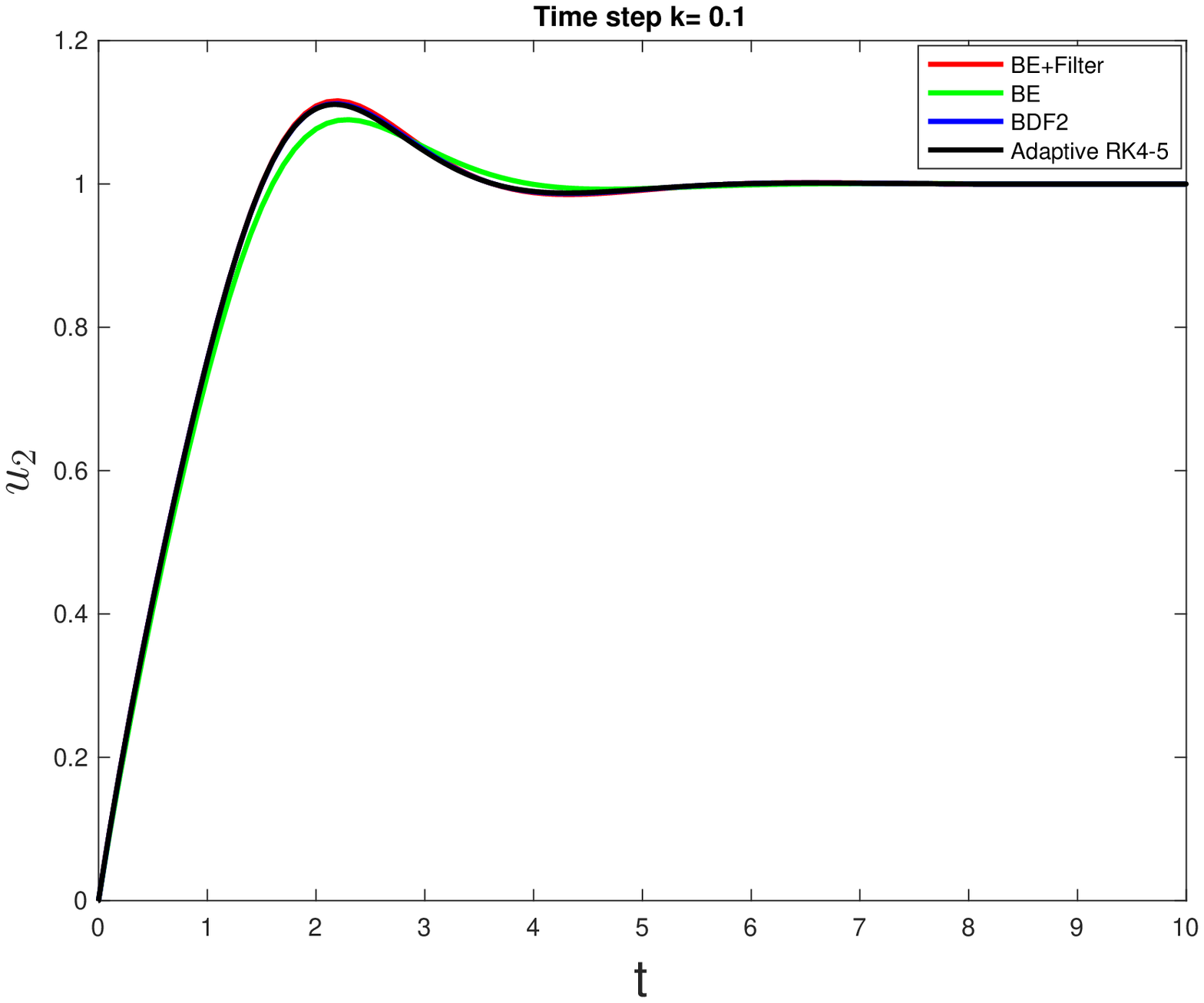} \hspace{-1.em}
	
	\includegraphics[width=6cm, height=6cm]{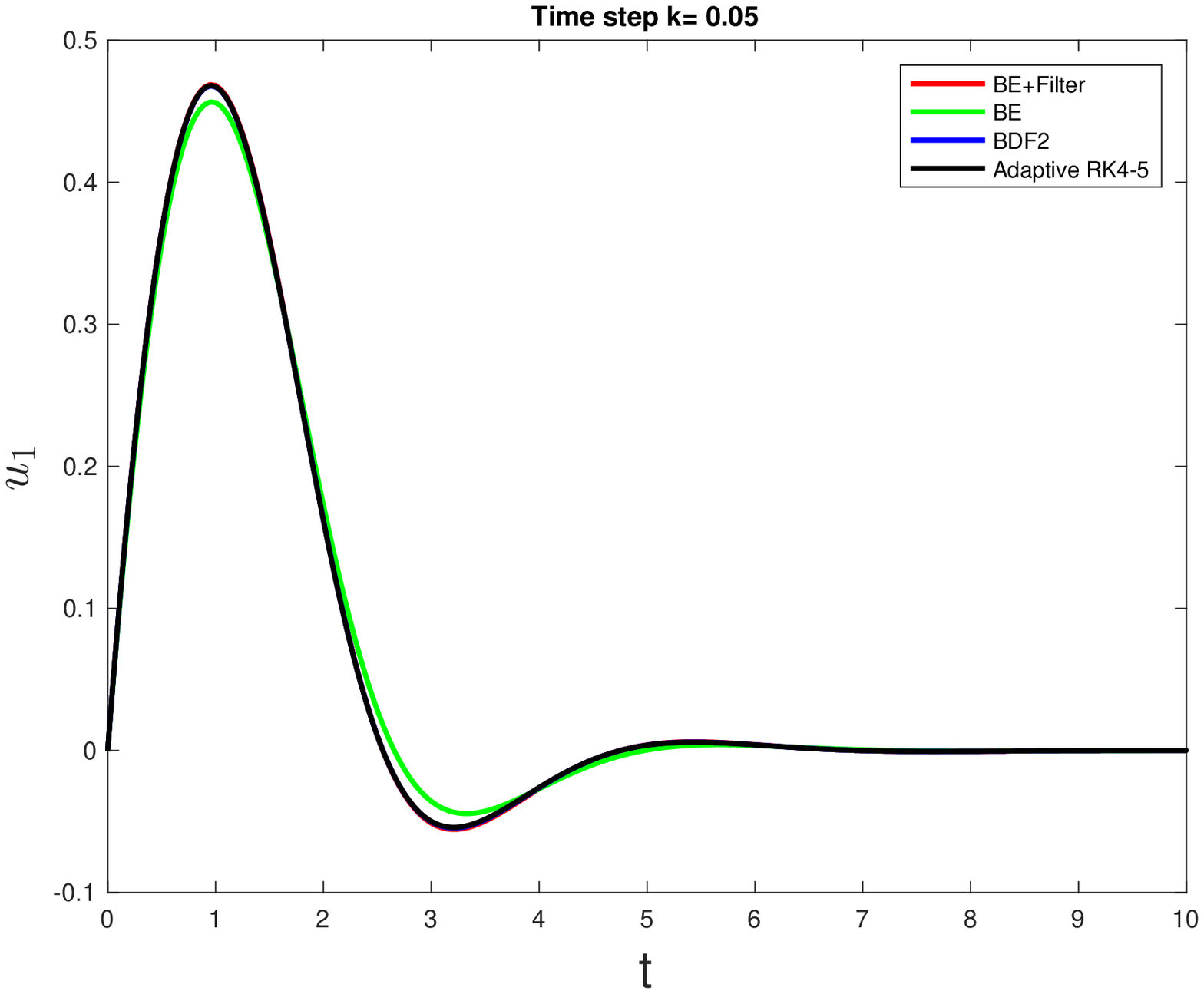}\hspace{-1.em}
	\includegraphics[width=6cm, height=6cm]{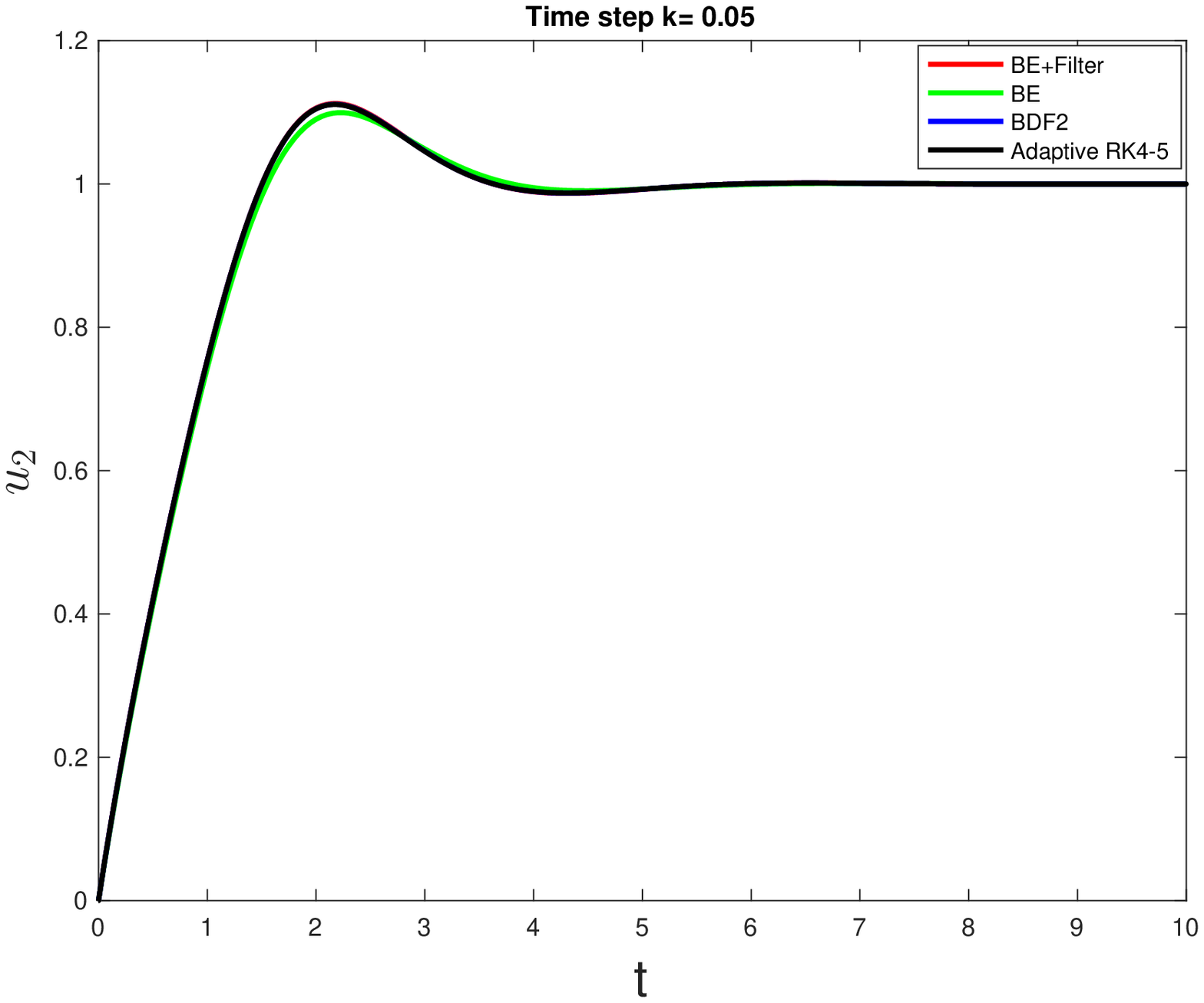}\hspace{-1.em}
	
	\caption{ The example of Sussman}
	\label{Fig:SussmanExample}
\end{figure}

In all cases, adding the filter step did not alter Lyapunov\ stability of
the equilibrium state.

\section{Conclusions}

While a satisfactory, variable timestep\ BDF2 method exists, the combination
of backward Euler plus a curvature reducing time filter gives another option
that is conceptually clear and\ easily added by one additional line to a
legacy code based on the implicit method. Both the theory and the tests both
show that adding the filter step to backward Euler greatly increases
accuracy.

\section{Appendix: 2 step methods} \hfill \\

The results were often developed by applying theory of $2-$step methods. We
collect here in this appendix some of the results applied. For constant
timestep, in the standard form of a $2-$step method is 
\begin{equation}
\alpha _{2}y_{n+1}+\alpha _{1}y_{n}+\alpha _{0}y_{n-1}=kf(t_{n+1},\beta
_{2}y_{n+1}+\beta _{1}y_{n}+\beta _{0}y_{n-1}).
\end{equation}%
This can be normalized in various ways; one normalization is to rescale so
the $\beta -$co\-efficients satisfy the standard normalization condition%
\begin{equation*}
\beta _{2}\ +\beta _{1}\ +\beta _{0}=1.
\end{equation*}%
The local truncation error is developed by expanding in a standard way in
Taylor series, giving%
\begin{gather*}
LTE=\left[ \alpha _{2}\ +\alpha _{1}\ +\alpha _{0}\right] y(t_{n})+k[\alpha
_{2}\ -\alpha _{0}\ -(\beta _{2}\ +\beta _{1}\ +\beta _{0}\ )]y^{\prime
}(t_{n}) \\
+k^{2}[\frac{\alpha _{2}}{2}+\frac{\alpha _{0}}{2}-\beta _{2}\ +\beta
_{0}]y^{\prime \prime }(t_{n})+k^{3}[\frac{\alpha _{2}}{6}-\frac{\alpha _{0}%
}{6}-\frac{\beta _{2}}{2}\ -\frac{\beta _{0}}{2}]y^{\prime \prime \prime
}(t_{n})+\mathcal{O}(k^{4}).
\end{gather*}%
A method is consistent if and only if the first two terms in the LTE
expansion are zero and second order accurate if and only if the third term
vanishes.

Next consider variable timesteps. The 2-step method for variable timestep is 
\begin{equation}
\alpha _{2}y_{n+1}+\alpha _{1}y_{n}+\alpha _{0}y_{n-1}=k_{n}f(t_{n+1},\beta
_{2}y_{n+1}+\beta _{1}y_{n}+\beta _{0}y_{n-1})
\end{equation}%
where the coefficients will depend on $\tau $ where%
\begin{equation*}
\tau =\frac{k_{n}}{k_{n-1}}\text{ and thus }k_{n}=\tau k_{n-1}.
\end{equation*}%
The LTE expansion for $t_{n+1}-t_{n}=k_{n}$, $t_{n}-t_{n-1}=k_{n-1}$ is now%
\begin{gather*}
LTE=\left[ \alpha _{2}\ +\alpha _{1}\ +\alpha _{0}\right] y(t_{n})+ \\
k_{n}[\alpha _{2}-\alpha _{0}\frac{1}{\tau }\ -(\beta _{2}\ +\beta _{1}\
+\beta _{0}\ )]y^{\prime }(t_{n})+k_{n}{}^{2}[\frac{1}{2}\alpha _{2}\ +\frac{%
1}{2\tau ^{2}}\alpha _{0}\ -\beta _{2}\ +\frac{1}{\tau }\beta _{0}]y^{\prime
\prime }(t_{n})+ \\
+k_{n}{}^{3}[\frac{1}{6}\alpha _{2}\ -\frac{1}{6\tau ^{3}}\alpha _{0}\ -%
\frac{1}{2}\beta _{2}\ -\frac{1}{2\tau ^{2}}\beta _{0}]y^{\prime \prime
\prime }(t_{n})+\mathcal{O}(k_{n}^{4})
\end{gather*}%
The method is consistent if and only if the first two terms are zero and
second order accurate if and only if the third term vanishes:%
\begin{eqnarray*}
Consistent &\Leftrightarrow &\left\{ 
\begin{array}{c}
\alpha _{2}+\alpha _{1}+\alpha _{0}=0\text{ }and \\ 
\\
\alpha _{2}\ -\frac{1}{\tau }\alpha _{0}\ -(\beta _{2}\ +\beta _{1}\ +\beta
_{0}\ )=0%
\end{array}%
\right. \\
\\
Second \text{ }order &\Leftrightarrow &\frac{1}{2}\alpha _{2}\ +\frac{1}{2\tau ^{2}}%
\alpha _{0}\ -\beta _{2}\ +\frac{1}{\tau }\beta _{0}=0.
\end{eqnarray*}

As defined by, e.g., Dahlquist, Liniger and Nevanlinna \cite{DLN83} equation
(1.12) p.1072, a variable step size method is $A-$stable if, when applied as
a one-leg scheme to 
\begin{equation*}
y^{\prime }=\lambda (t)y,\text{ } \func{Re}(\lambda (t))\leq 0,
\end{equation*}%
solutions are always bounded for any sequence of step sizes. Conditions for
variable stepsize, $A-$stability were derived for variable step, 2-step
methods in Dahlquist \cite{D79}. The characterization in Dahlquist \cite{D79}, Lemma 4.1 page 3,
4 (specifically rearranging the equation on page 4 following (4.1)), states
that the method is $A-$stable if%
\begin{equation}
\begin{cases}
-\alpha _{1}\geq 0,\\
\\
1-2\beta _{1}\geq 0\text{ and }\\
\\
2(\beta _{2}\ -\beta
_{0})+\alpha _{1}\geq 0.
\end{cases}
\end{equation}

\textbf{Adaptivity. }The combination of backward Euler plus filter lends
itself to adaptive implementation. There are various choices that must be
made in such an implementation. We have purposefully made the simplest one
of each option. With simple timestep halving and doubling the general
adaptive method implemented was as follows.%
\begin{eqnarray*}
\text{Given}\text{: } &&y_{n},y_{n-1},k_{n-1},k_{n}\text{ and }Tol \\
\text{Choose}\text{: } &&\nu =\frac{k_n(k_n+k_{n-1})}{k_{n-1}(2k_n+k_{n-1})\ } \\
\text{Compute}\text{: } &&y_{n+1}^{prefilter},y_{n+1}^{postfilter}\text{ by}
\\
\frac{y_{n+1}-y_{n}}{k_{n}} &=&f(t_{n+1},y_{n+1}), \\
y_{n+1} &\Leftarrow &y_{n+1}-\frac{\nu }{2}\left\{ \frac{2k_{n-1}}{%
k_{n}+k_{n-1}}y_{n+1}-2y_{n}+\frac{2k_{n}}{k_{n}+k_{n-1}}y_{n-1}\right\} \\
EST &=&|y_{n+1}^{prefilter}-y_{n+1}^{postfilter}| \\
\text{If}\text{: } &&EST\geq Tol\text{, }k_{n}\Leftarrow
k_{n}/2\text{\ and repeat step} \\
\text{ If}\text{: } &&EST\leq \frac{Tol}{8}\text{, }%
k_{n+1}=2k_{n}\text{\ and next step} \\
\text{Else}\text{: } &&k_{n+1}=k_{n}\text{\ and next step }
\end{eqnarray*}

\end{document}